\RequirePackage{fix-cm}

\documentclass[twocolumn,numbook]{svjour3}          % twocolumn
\smartqed  % flush right qed marks, e.g. at end of proof
\usepackage{graphicx}
\usepackage{amsfonts, amsmath, amssymb, url,dsfont}
\usepackage{graphicx}
\usepackage[english] {babel}
\usepackage{enumitem}
\usepackage{array}
\DeclareMathOperator{\pos}{pos}
\DeclareMathOperator{\Pol}{Pol}

\newcommand{\R}{\mathbb{R}}

\newcommand{\N}{\mathbb{N}}

\newcommand{\Z}{\mathbb{Z}}
\newcommand{\La}{\mathbb{L}}

\newcommand{\eps}{\varepsilon}
\DeclareMathOperator{\conv}{conv}

\DeclareMathOperator{\su}{sup}

\DeclareMathOperator{\indre}{int}

\DeclareMathOperator{\tr}{Tr}
\newcommand{\Ha}{\mathcal{H}}

\newcommand{\II}{\textit{II}}
\DeclareMathOperator{\lin}{lin}
\DeclareMathOperator{\aff}{aff}

\newcommand*\InsertTheoremBreak{
\begingroup % keep changes local
\setlength\itemsep{0pt}\setlength\parsep{0pt}
\item[\vbox{\null}]
\endgroup
}

\begin{document}

\title{On multigrid convergence of local algorithms for intrinsic volumes}

\author{Anne Marie Svane}

\institute{Anne Marie Svane \at Department of Mathematics, Aarhus University, 8000 Aarhus C, Denmark \\ \email{amsvane@imf.au.dk}} 
\date{}
\maketitle

\begin{abstract}
Local digital algorithms based on $n\times \dots \times n$ configuration counts are commonly used within science for estimating intrinsic volumes from binary images. This paper investigates multigrid convergence of such algorithms. It is shown that local algorithms for intrinsic volumes other than volume are not multigrid convergent on the class of convex polytopes. In fact, counter examples are plenty. On the other hand, for convex particles in 2D with a lower bound on the interior angles, a multigrid convergent local algorithm for the Euler characteristic is constructed. Also on the class of $r$-regular sets, counter examples to multigrid convergence are constructed for the surface area and the integrated mean curvature. 

\end{abstract} 

\section{Introduction and main results}\label{intro}
The purpose of this paper is to assess a certain class of algorithms that are widely used for analysing digital output data from e.g.\ microscopes and scanners. These algorithms yield a fast way of estimating the so-called intrinsic volumes $V_q$, $q=0,\dots ,d$, of a given object. The intrinsic volumes include many of the quantities, scientists are most frequently interested in, see e.g.\ \cite{mecke2}, such as the volume $V_d$, the surface area $2V_{d-1}$, the integrated mean curvature $2\pi(d-1)^{-1}V_{d-2}$, and the Euler characteristic $V_0$. 

The algorithms considered rely only on what the image looks like locally, thus we refer to them as local algorithms. The use of local algorithms goes back to \cite{freeman}, see also \cite{digital,lindblad} for an overview of the algorithms suggested in the literature. The popularity of local algorithms is due to the fact that they allow simple linear time implementations \cite{OM}, as opposed to the more complex algorithms of \cite{coeurjolly,sun}. However, as we shall see below, this efficiency is often paid for by a lack of accuracy.

We model a digital image of an object $X\subseteq \R^d$ by a binary image, i.e.\ as the set $X\cap \La$ where $\La$ is some lattice in $\R^d$. In applications, such a binary image is usually obtained from an observed grey-scale image by thresholding. Each point in $\La$ may belong to either $X$ or its complement. 
For every $n\times \dots \times n$ cell in the observation lattice, this yields $2^{n^d}$ possible configurations of foreground and background points. The idea of local algorithms is to estimate $V_q$ as a weighted sum of configuration counts, see Definition \ref{funddef}. 

Local algorithms are suggested many places in the li\-te\-ra\-ture \cite{digital,lindblad,turk,mecke} and various partial definitions are given \cite{kampf2,klette,digital,am2}. In Section \ref{local} we attempt to set up a unified, rigorous definition of local algorithms  and, in particular, to justify the use of local algorithms for the estimation of intrinsic volumes. 

The next question is, when a local algorithm yields a good approximation of $V_q$. A natural criterion for an algorithm is multigrid convergence, i.e.\ that the estimator converges to the true value when the resolution goes to infinity. This is a very strong and in applications often unnatural requirement. Instead, a design based setting is considered where the lattice has been randomly translated before making the observation. The natural, and usually weaker, requirement in this situation is that the estimator should be unbiased, at least asymptotically when the resolution tends to infinity. The various convergence criteria are discussed in Section \ref{convergence} in more detail.

In order for the digital image to contain enough information about $X$ to enable us to estimate $V_q(X)$, some niceness assumptions on the underlying set $X$ are needed. In this paper, we shall investigate which intrinsic volumes $V_q$ allow asymptotically unbiased local estimators when $X$ is assumed to belong to the class of compact convex polytopes with non-empty interior or the class of $r$-regular sets (see Definition~\ref{defrreg}).   

\subsection{Known results}\label{known}
Various results have already been obtained in this direction.
It is well-known, see e.g. \cite{OM}, that there is a local estimator for the volume $V_d$ which is unbiased even in finite resolution given by counting lattice points in $X$ and weighting them by the volume of the unit lattice cell. 

In contrast, J\"{u}rgen Kampf has proved \cite{kampf2} that on the class of finite unions of polytopes, local algorithms for $V_q$ based on $2\times \dots \times 2$ configurations in orthogonal lattices are always asymptotically biased for $0 \leq q  \leq d-2$. In fact, he has shown that the worst case asymptotic bias is always at least $100\%$. 

For $q=d-1$, Ziegel and Kiderlen showed in \cite{johanna} that there exists no asymptotically unbiased local algorithm for the surface area in 3D based on $2\times 2 \times 2$ configurations in an orthogonal lattice, but the asymptotic worst case bias is strictly less than 100\% in this case. 

It has been conjectured in \cite{digital} and \cite{klette} that no local algorithm for estimation of surface area is multigrid convergent in dimension $d=2$ and $d=3$, respectively. This was proved by Tajine and Daurat \cite{tajine} in dimension $d=2$ in the special case of length estimation for straight line segments. In fact, they show that any algorithm will be (asymptotically) biased for almost all slopes of the line segment. In \cite[Theorem~5]{rataj}, Kiderlen and Rataj prove a formula for the asymptotic mean of a surface area estimator, on which a proof in arbitrary dimension $d$ could be based.

On the other hand it is known that with suitable smoothness conditions ($r$-regularity) on the boun\-da\-ry $\partial X$ there exists  a multigrid convergent local algorithm for estimating the Euler characteristic $V_0$ in 2D \cite{pavlidis} and in 3D, see \cite{SLS} combined with \cite{lee}. In fact, this algorithm yields the correct value in sufficiently high finite resolution. The Euler characteristic differs from the other intrinsic volumes in that it is a topological invariant, hence it only requires a topologically correct reconstruction of $X$ to estimate it.
It is still a partially open question whether the remaining $V_q$ can be estimated on the class of $r$-regular sets. However, it is shown in \cite{am2} that there is no asymptotically unbiased estimator for the integrated mean curvature $V_{d-2}$ in dimensions $d>2$ based on $2\times \dots \times 2$ configurations. This suggests that $V_0$ is special.

\subsection{Main results of the paper}\label{main}
We first consider the estimation of $V_q$ on the class $\mathcal{P}^d$ of compact convex polytopes  with non-empty interior. Any $P\in \mathcal{P}^d$ can be written in the form
\begin{equation*}
P=\bigcap_{i=1}^N H^-_{u_i,t_i}
\end{equation*}
where $H^-_{u,t}$ denotes the halfspace $\{x\in \R^d \mid \langle x, u \rangle \leq t\}$ where $\langle \cdot,\cdot \rangle$ denotes the standard Euclidean inner pro\-duct, $u$ belongs to the unit sphere $S^{d-1}$ and $t\in \R$. The parameters $u_i,t_i$ can be used to define a measure $\nu$ on $\mathcal{P}^d$. This is made precise in Section \ref{polspace}. 

When $1\leq q \leq d-1$, we shall prove the following theorem:
\begin{theorem}\label{main1}
For $1\leq q \leq d-1$, any local algorithm for ${V}_q$ in the sense of Definition \ref{funddef} is asymptotically biased (and hence not multigrid convergent) for $\nu $-almost all $P\in \mathcal{P}^d$ if $d-q$ is odd and for a subset of $\mathcal{P}^d$ of positive $\nu$-measure if $d-q$ is even.
\end{theorem}
This generalizes the results of~\cite{tajine} to dimensions $d>2$ and the results of \cite{kampf2} to $n\times \dotsm \times n$ configurations with $n>2$ in general lattices and with the sets for which an asymptotic bias occur chosen from an even smaller set class. As simple examples, one may take almost all rotations of almost all orthogonal boxes $\bigoplus_{i=d}^d [0,t_ie_i]$ where $t_1,\dots ,t_d\in \R$ and $e_1,\dots,e_d\in \R^d$ is the standard basis. 

If an algorithm were only asymptotically biased for a very small class of sets, for instance orthogonal boxes, this could well be acceptable in practice where objects are often randomly shaped with a probability of zero for hitting this class. Hence the theorem is stated for all polytopes in a set of positive $\nu$-measure. The reasonableness in choosing the measure  $\nu $ on $\mathcal{P}^d$ may be disputed, see the discussion in Section \ref{polspace}. 

In the case $q=0$, we can obtain a similar theorem, again generalizing the results of \cite{kampf2}:
\begin{theorem}\label{main2}
Any local algorithm for ${V}_0$ in the sense of Definition \ref{funddef} is asymptotically biased (and hence not multigrid convergent) on $\mathcal{P}^d$ if $d>1$.
\end{theorem}
However, constructing counter examples is now harder. 
In fact, in $\R^2$ there is a sequence of local algorithms $\hat{V}_0^n$  for $n\in \N$ based on $n\times \dots \times n$ configurations such that $\hat{V}_0^n$ is multigrid convergent for all $P\in \mathcal{P}^d$ (or even all compact convex sets having interior points) having no interior angles less than $\psi_n\in \R$ where $\lim_{n\to 0} \psi_n=0$. In particular, for any $P\in \mathcal{P}^d$ there is an $N\in \N$ such that $\hat{V}_0^n(P)=V_0(P)$ whenever $n\geq N$ and the resolution is sufficiently high. Thus, if one studies convex particles with a lower bound on the interior angles, there exists a multigrid convergent local algorithm for $V_0$.  The explicit construction of these algorithms and the precise conditions on the weights are given in Section \ref{euler}.

As in \cite{kampf2}, the proof of Theorem \ref{main2} goes by first constructing a counter example $P\subseteq \R^2$ and then generalizing this to higher dimensions by means of the prism $P \times \bigoplus_{i=3}^d [0,e_i]$. This approach also provides the following generalization of Kampf's results:
\begin{theorem}\label{main3}
For $0\leq q\leq d-2$, any local algorithm for ${V}_q$ as in Definition \ref{funddef} has an asymptotic worst case bias of at least 100\% on $\mathcal{P}^d$.
\end{theorem}

We finally move on to the case of $r$-regular sets. Using the main results of \cite{rataj} and \cite{am2}, we show the following generalization of \cite{am2} to $n>2$ and arbitrary lattices:
\begin{theorem}\label{mainrreg}
For $q=d-1$ and, if $d\geq 3$, also for $q=d-2$, any local algorithm for ${V}_{q}$ as in Definition \ref{funddef} with homogeneous weights is asymptotically biased (and hence not multigrid convergent) on the class of $r$-regular sets.
\end{theorem}
The definition of homogeneous weights is given in De\-fi\-ni\-tion \ref{weightcond} below.
For $0<q<d-2$, the asymptotic behavior of local estimators for $V_q$ is not well enough understood to determine whether asymptotically unbiased estimators exist. However, Theorem \ref{mainrreg} suggests that the Euler characteristic is the only $V_q$ with $q<d$ that allows an asymptotically unbiased local estimator on the class of $r$-regular sets.

\section{Local digital algorithms} \label{local}

%In Section \ref{convergence} we discuss various notions of convergence of such algorithms. Local algorithms are introduced in general in Section \ref{lda} and in Section \ref{localal} this is specialized to estimators for intrinsic volumes.

\subsection{Digital estimators}\label{de}
We first set up some notation and terminology and introduce digital estimators in general.

Let $\xi=\{\xi_1,\dots,\xi_d\}$ be a positively oriented basis of $\R^d$ and let $\La$ denote the lattice spanned by $\xi$.
Let $C_\xi=\bigoplus_{i=1}^d [0,\xi_i]$ be the unit cell of the lattice with volume $\det(\La)$. For $c\in \R^d$, we let $\La_c=\La+c$ denote the translated lattice.

Now suppose $X\subseteq \R^d$ is some subset of $\R^d$. We use the binary digitization model for a digital image, see e.g.\ \cite{OM}. That is, we think of a digital image as the set $X\cap a\La_c \subseteq a\La_c$ where $a>0$ is the lattice distance. This set contains the same information about $X$ as the Gauss digitization \cite[Definition 2.7]{digital}, which is the union of all translations of $C_\xi$ having midpoint in $X\cap a\La_c$.

Let $V:\mathcal{S} \to \R $ be a function defined on some class $\mathcal{S}$ of subsets of $\R^d$. We want to estimate this function based on digital images of elements of $\mathcal{S}$. 
\begin{definition}\label{digaldef}
By a digital algorithm $\hat{V}$ for $V$, we mean a collection of functions $\hat{V}^{a\La_c}: \mathcal{P}(a \La_c )\to \R$ for every $a>0$ and $c\in C_\xi$ where $\mathcal{P}(a \La_c)$ is the power set of $a \La_c$. 
For $X\in \mathcal{S}$ we use $\hat{V}^{a\La_c}(X):=\hat{V}^{a\La_c}(X\cap a\La_c)$ as a digital estimator for $V(X)$.

A digital algorithm $\hat{V}$ is said to be 
\begin{itemize}
\item \textbf{translation invariant} if 
\begin{equation*}
\hat{V}^{a\La_0}(S)=\hat{V}^{a\La_c}(S+ac+az)
\end{equation*}
 for all $S\in \mathcal{P}(a\La)$, $c\in C_\xi$, $z\in \La$, and $a>0$.
\item \textbf{rotation (reflection) invariant} if 
\begin{equation*}
\hat{V}^{a\La_c}(S)=\hat{V}^{a\La_{Rc}}(RS)
\end{equation*}
for all $S\in \mathcal{P}(a\La)$, $c\in C_\xi$, $a>0$, and all rotations (reflections) $R\in SO(d)$ preserving $a\La$.
\item \textbf{motion invariant} if it is both translation and rotation invariant.
\end{itemize}
\end{definition}

\begin{remark}
Sometimes, e.g.\ in \cite{serra}, $\hat{V}^{a\La}$ is only defined for $a$ belonging to some sequence $a_k\to 0$ (typically, $a_k=2^{-k}$). 
Though a weaker requirement, this will not affect the non-existence theorems of this paper, so we consider only the case of the definition. 

Similarly, the algorithm is sometimes only defined for a subset of $\mathcal{P}(a \La_c )$, e.g.\ finite sets, or only for $c=0$, but of course, such a definition can easily be extended.
\end{remark}

\subsection{Various convergence criteria}\label{convergence}
Having defined a digital algorithm, the next question is how it should relate to $V(X)$.
Obviously, many different sets may have the same digital image, so $\hat{V}^{a\La_c}(X)$ will typically not give the correct value. However, $X\cap a\La_c$ will contain more and more information about $X$ as $a$ decreases. Thus it is reasonable to require that $\hat{V}^{a\La_c}(X)$ converges to the correct value when the lattice distance goes to zero.
In \cite{digital}, this is called multigrid convergence and the formal definition here is as follows:
\begin{definition}\label{mc}
A digital algorithm $\hat{V}$ for $V: \mathcal{S} \to \R $ is called multigrid convergent if for all $X\in \mathcal{S}$,
\begin{equation*}
\lim_{a\to 0} \hat{V}^{a\La_0}(X)={V}(X).
\end{equation*}
\end{definition}
Note that the definition only involves the non-trans\-la\-ted lattice $\La_0$.
This definition does cause some problems. It depends on the choice of origin with respect to which the lattice is scaled. For instance, it could be that $\hat{V}^{a\La_c}(X)$ does not converge to $V(X)$, even if the algorithm is translation invariant. One could of course repair this by requiring $\lim_{a\to 0} \hat{V}^{a\La_c}(X)={V}(X)$ for all $c\in C_\xi$. Thus the following stronger condition would be natural:
\begin{definition}
A digital algorithm $\hat{V}$ is called uniformly multigrid convergent if for all $X\in \mathcal{S}$ and $\eps>0$ there is a $\delta>0$ such that 
\begin{equation*}
|\hat{V}^{a\La_c}(X)-{V}(X)| \leq \eps 
\end{equation*}
for all $c\in C_\xi$ and $a < \delta$. 
\end{definition}
In other words, the convergence $\hat{V}^{a\La_c}(X)\to {V}(X)$ is uniform with respect to translations of $\La$. An e\-qui\-va\-lent formulation is that for every pair of sequences $a_k \to 0^+$ and $c_k \in \R^d$, 
\begin{equation*}
\lim_{k\to \infty} \hat{V}^{a_k\La_{c_k}}(X) = {V}(X).
\end{equation*}

Multigrid convergence is in many situations a much too strong requirement. Of the examples mentioned in the introduction, only the volume estimator on the class $\mathcal{C}_\partial$ defined below and the estimator for the Euler cha\-rac\-te\-ris\-tic of $r$-regular sets is multigrid convergent.

Another way of removing the dependence on the origin is to consider a uniform random translation of the lattice. This is called the design based setting and the observed image is now a random set $X\cap a\La_c$ where $c\in C_\xi$ is a uniform random translation vector. A di\-gi\-tal algorithm is called integrable if $c\mapsto \hat{V}^{a\La_c}(X)$ is integrable over $C_\xi$ for all $a>0$ and $X\in \mathcal{S}$, i.e.\ the mean $E\hat{V}^{a\La_c}(X)$ is finite for all $X\in \mathcal{S}$. The natural requirement for an integrable digital algorithm is that $\hat{V}^{a\La_c}(X)$ is unbiased, at least when $a$ tends to zero. More formally:
\begin{definition}
Let $\hat{V}$ be an integrable digital algorithm for $V$ defined on a class $\mathcal{S}$ of subsets of $\R^d$. Then $\hat{V}$ is called asymptotically unbiased if for all $X \in \mathcal{S}$, 
\begin{equation*}
\lim_{a\to 0}E\hat{V}^{a\La_c}(X) = V(X).
\end{equation*}
\end{definition}

It is clear that uniform multigrid convergence implies asymptotic unbiasedness. So does multigrid convergence in most nice situations, as the next proposition shows.
Let $\mathcal{C}_\partial$ denote the collection of compact subsets of $\R^d$ whose boundary has $\Ha^{d}$-measure zero where $\Ha^{k}$ denotes the $k$-dimensional Hausdorff measure.
\begin{proposition}
Suppose $V: \mathcal{S} \to \R$ is a translation invariant function defined on some $\mathcal{S}\subseteq \mathcal{C}_\partial $ and that $\hat{V}^{a\La_c}$ is a translation invariant digital estimator for $V$.  Then multigrid convergence implies asymptotic unbiasedness.
\end{proposition}

\begin{proof}
Suppose $X\in \mathcal{S}$ and that $\hat{V}$ is multigrid convergent.
It will be enough to show that for all $\eps>0$ there is a $\delta>0$ such that for all $a<\delta $,
\begin{equation*}
|\hat{V}^{a\La_0}((X-ac)\cap a\La)-{V}(X)| < \eps 
\end{equation*}
holds for almost all $c\in C_\xi$. 

Assume this were not true. Then there would be an $\eps>0$, a sequence $a_m \to 0$, and $W_m \subseteq C_\xi$ with $\Ha^{d}(W_m)>0$ such that
\begin{equation*}
|\hat{V}^{a_m\La_0}((X-a_mc)\cap a_m\La)-{V}(X)| \geq \eps
\end{equation*}
for all $c\in W_m$.

First assume that $a$ is fixed. By compactness of $X$, $ (X-ac)\cap a\La $ can take only finitely many values in $\mathcal{P}(a\La)$ when $c\in C_\xi$. Thus also $\hat{V}^{a\La_0}((X-ac)\cap a\La)$ takes only finitely many different values for $c\in C_\xi$. 

Define 
\begin{equation*}
S_{z}=\{c\in C_\xi \mid az \in X -ac \} = C_\xi \cap (a^{-1}X-z)
\end{equation*}
for $z\in \La$ and note that only finitely many $S_z$ are non-empty. Thus for $S\subseteq a\La$
\begin{equation}\label{Sz}
\{c\in C_\xi \mid (X-ac)\cap a\La =S \} = \bigcap_{z\in S}S_z \cap \bigcap_{z\notin S}S_z^c  
\end{equation}
Observe that $ S_z^c \cap \indre C_\xi$ is open and equals $\indre C_\xi$ for all but finitely many $z$. 
The boundary of $S_z$ is contained in $\partial C_\xi \cup \partial (a^{-1}X-z)$ and therefore it has $\Ha^d$-measure zero. 
A point in \eqref{Sz} will either lie in the interior of all $S_z$, $z\in S$, or in the boundary of one of them. Thus~\eqref{Sz} will either have non-empty interior or $\Ha^d$-measure zero.

Since $W_m$ is the finite union of sets of the form \eqref{Sz} and $\Ha^d(W_m)>0$, it must have non-empty interior $U_m$.
Now choose $a_{m_{i}}$ inductively. First let $a_{m_{1}}=a_{1}$ and let $K_{m_1}\subseteq U_1$ be a compact set with non-empty interior. For $a_{m_2}$ sufficiently small, $a_{m_2}(C_\xi +z) \subseteq K_{m_1}$ for some $z$. Therefore we may choose a compact set with non-empty interior $K_{m_{2}}\subseteq  K_{m_{1}}\cap a_{m_{2}}(U_{m_{2}}+z)$. Continuing this way yields a decreasing sequence of compact sets $K_{m_i}$. In particular, $\bigcap K_{m_{i}}$ is non-empty, so we may choose $y\in \bigcap K_{m_{i}}$. By the translation invariance of $\hat{V}^{a\La_0}$  and $V$,
\begin{equation*}
|\hat{V}^{a_{m_i}\La_0}((X-y)\cap a_{m_i}\La )-V(X)|\geq \eps
\end{equation*}
for all $i$, so $\hat{V}$ is not multigrid convergent for $X-y$, which is a contradiction. \qed
\end{proof}

\subsection{Local digital algorithms}\label{lda}
In this section we introduce the notion of local algorithms. The name `local algorithm' is adopted from \cite[Definition 4.1]{klette} and \cite[Definition 8.3]{digital}. In these definitions, a local algorithm is really an algorithm for reconstructing the boundary of a solid in 2D or 3D as a union of line segments or polygons, respectively. The idea is that each of these building blocks should only depend on what the digital image looks like locally. From the reconstructed set, the length or surface area can be estimated as a sum of lengths or areas of the building blocks, respectively. The authors also refer to algorithms for estimating length and surface area arising in this way as local algorithms. 

We choose the following definition for general digital algorithms:
\begin{definition}\label{defgen}
A digital algorithm $\hat{V}$ is called local if there is a finite collection of pairs $(B_k,W_k)$ for $k\in K$ such that $B_k,W_k\subseteq \La$ are two finite disjoint sets and
\begin{align}
\begin{split}\label{Vgen}
\hat{V}^{a\La_c}(S)={}&  \sum_{k\in K}  \sum_{z\in \La} {w}_k(a,a(z+c)) \\
&\times  \mathds{1}_{\{a(B_k+z+c) \subseteq S, a(W_k+z+c) \subseteq  a\La_c \backslash S \} }
\end{split}
\end{align}
for all finite $S\subseteq a\La_c$. Here $\mathds{1}_A$ denotes the indicator function for the set $A$. The pair $(B_k,W_k)$ is called a configuration and the elements of $B_k$ are referred to as the `foreground' or `black' pixels, while $W_k$ is referred to as the set of `background' or `white' pixels. The functions ${w}_k:(0,\infty) \times \R^d \to \R$ are called the weights.
\end{definition}
Thus each occurrence of a translation of the configuration $(B_k,W_k) $ contributes to the estimate with a weight $w_k(a,z)$ depending only on the translation vector $z$ and the lattice distance $a$.
%The elements of $B_k$ are referred to as the `foreground' or `black' pixels, while the complement $W_k=C_{z,0}^n\backslash B_k$ is referred to as the `background' or `white' pixels. 

The definitions of \cite{klette} and \cite{digital} correspond to the collection
\begin{equation*}
\{(B_k,W_k) \mid  B_k \cup W_k = B({R}) \cap \La, B_k\cap W_k =\emptyset\}
\end{equation*}
for some $R>0$ where $B(R)$ denotes the ball of radius $R$.
Strictly speaking, their definition is not quite contained in Definition \ref{defgen}. However, all the examples of local algorithms for computing length and surface area mentioned in these references are of this form. 

We introduce a bit more notation:
An $n\times \dots \times n$ cell is a set of the form $C_z^n=(z+\bigoplus_{i=1}^d [0,n\xi_i))$ for $z\in \La$. The set of lattice points lying in such a cell is denoted by $C_{z,0}^n=C_z^n\cap \La $. A lattice point in $C_{0,0}^n$ has the form $x=\sum_{i=1}^d\lambda_i \xi_i$ for some $\lambda_i\in \{0,\dots ,n-1\}$ and we write $x=x_j$ where the index is given by
\begin{equation*}
j=\sum_{i=1}^d n^{i-1}\lambda_i.
\end{equation*}

An $n\times \dotsm \times n$ configuration is a pair $(B^n,W^n)$ where $B^n, W^n \subseteq C_{0,0}^n$ are disjoint with $B^n\cup W^n = C_{0,0}^n$. 
We index these by $(B_l^n,W_l^n)$, $l=0,\dots,2^{n^d}-1$, 
where a configuration $(B^n,W^n)$ is assigned the index 
\begin{equation*}
l=\sum_{i=0}^{{n^d}-1} 2^i \mathds{1}_{\{x_i\in B^n\}}.
\end{equation*}  

%
%A local estimator can now be written in the form
%\begin{equation*}
%\hat{V}^{a\La}(S)= \sum_{z\in \La} w_l(a,z)\mathds_{\{(S-az)\cap C_{0,0}^n=B_l\}}.
%\end{equation*}
%The numbers $w_l(a,z)\in \R$ are called the weights.

\begin{proposition}
For every local algorithm $\hat{V}$ there is an $n\in \N$ such that for all finite $S\subseteq a\La_c$, 
\begin{equation}
\begin{split}\label{Vnconf}
\hat{V}^{a\La_c}(S)= {}& \sum_{l=0}^{2^{n^d}-1} \sum_{z\in \La} \tilde{w}_l(a,a(z+c))\\ &\times \mathds{1}_{\{a(B_l^n + z+ c)\subseteq S, a(W_l^n + z +c ) \subseteq \R^d\backslash S\}}
\end{split}
\end{equation}
for suitable weights $\tilde{w}_l(a,z)$.
\end{proposition}

\begin{proof}
By finiteness of $K$, there is an $n\in \N$ and a $y\in \La$ with $B_k,W_k \subseteq C_{y,0}^n$ for all $k\in K$. Thus, \eqref{Vgen} becomes an estimator of the form \eqref{Vnconf} with weights
\begin{align*}
w_l(a,z) ={}& \sum_{k\in K} {w}_k(a,z) \mathds{1}_{\{B_k-y\subseteq B_l^n, W_k-y \subseteq W_l^n\}}\\
& \times
\bigg(\sum_{m=0}^{2^{n^d}-1} \mathds{1}_{\{B_k-y\subseteq B_m^n, W_k-y \subseteq W_m^n\}}\bigg)^{-1}.
\end{align*}\qed
\end{proof}

Thus, for the remainder of this paper we shall only consider local algorithms of the form \eqref{Vnconf}. We usually skip the $n$ from the notation and write $(B_l,W_l)$ for the $n\times \dotsm \times n$ configurations.

Clearly, the larger $n$ is, the better accuracy of the algorithm can be expected, as more information is taken into account.
For most algorithms used in practice \cite{digital,OM}, $n=2$. However, algorithms with $n=3$ have been suggested, see \cite{sandau}. Also, most theoretical studies of local algorithms only involve $n=2$, see Section \ref{known}. One exception is \cite{tajine}.

\begin{definition}\label{weightcond}
The weights are said to be
\begin{itemize}
\item \textbf{translation invariant} if $w_l(a,z)$ is independent of $ z \in \R^d$.
\item \textbf{rotation (reflection) invariant} if 
\begin{equation*}
w_{l_1}(a,z_1)=w_{l_2}(a,z_2)
\end{equation*}
 whenever there is a rotation (reflection) $R$  preserving $\La$ such that $R(B_{l_1}+ z_1)=B_{l_2}+ z_2$.
\item \textbf{motion invariant} if the weights are both translation and rotation invariant.
\item \textbf{homogeneous (of degree $q$)} if 
\begin{equation*}
w_l(a,z)=a^qw_l(1,z)
\end{equation*}
 for all $a>0$ and $ z \in \R^d$.
\end{itemize}
\end{definition}

The estimators for Minkowski tensors in e.g.\ \cite{turk,mecke} are examples of local digital estimators where the weights are not translation invariant. 
If $V$ is rotation (reflection) invariant, the following proposition justifies the choice of rotation (reflection) invariant weights, see also \cite{am2}:

\begin{proposition}\label{rotprop}
Assume $V$ is rotation (reflection) invariant. For every local algorithm $\hat{V}$, there is a local algorithm $\hat{W}$ with rotation (reflection) invariant weights such that for all compact $X \in \mathcal{S}$,
\begin{equation}
\begin{split}
\label{rotworst}
&\su_{R\in \mathcal{R}} |\hat{W}^{a\La_{Rc}}(RX)- V(RX)| \\
&\quad \leq \su_{R\in \mathcal{R}} |\hat{V}^{ a\La_{Rc}}(RX) - V(RX)|
\end{split}
\end{equation}
where $\mathcal{R}$ denotes the  group of rotations (reflections) preserving $\La$. 
\end{proposition}

\begin{proof}
If $|\mathcal{R}|$ is the cardinality of $\mathcal{R}$, define for $S\subseteq \La$
\begin{equation*}
\hat{W}^{ a\La_c}(a(S+c))=\frac{1}{|\mathcal{R}|}\sum_{R\in \mathcal{R}} \hat{V}^{ a\La_{Rc}}(aR(S+c)).
\end{equation*}
This is a local estimator with rotation invariant weights and it clearly satisfies \eqref{rotworst} since $V(RX)=V(X)$. \qed
\end{proof}

Finally, we introduce a bit more notation:
For two subsets $A,B\subseteq \R^d$, let 
\begin{align*}
&\check{B}=\{-b \mid b \in B\},\\
&A\ominus B=\{x\in\R^d \mid x + \check{B}\subseteq A\}.
\end{align*}
The hit-or-miss transform of $X$ with structure elements $B$ and $W$ is defined to be the set
\begin{equation*}
X\ominus \check{B} \backslash X\oplus \check{W}= \{y\in \R^d \mid y+B\subseteq X, y + W \subseteq \R^d \backslash X\}.
\end{equation*}
A local estimator then takes the form
\begin{align*}
\hat{V}^{a\La_c}(X)={}&\sum_{l=0}^{2^{n^d}-1} \sum_{z\in \La}  w_l(a,a(z+c))\\
& \times \mathds{1}_{  X\ominus a\check{B}_l \backslash X\oplus a\check{W}_l}(a(z+c)).
\end{align*}
If $z\mapsto w_0(a,z)$ is integrable and $z\mapsto w_l(a,z)$ are locally integrable for $l>0$, $\hat{V}^{a\La_c}(X)$ is always integrable for $X$ compact since
\begin{equation}
\begin{split}\label{intHM}
E {}&\bigg( \sum_{z\in \La}  w_l(a,a(z+c))\mathds{1}_{  X\ominus a\check{B}_l \backslash X\oplus a\check{W}_l}(a(z+c))\bigg) \\ %\nonumber
&= a^{-d}\det(\La)^{-1} \int_{X\ominus a\check{B}_l \backslash X\oplus a\check{W}_l} w_l(a,z)dz
\end{split}
\end{equation}
and hence
\begin{align*}
E{}&\hat{V}(X\cap a\La_c)\\
& =a^{-d}\det(\La)^{-1}\sum_{l=0}^{2^{n^d}-1}  \int_{X\ominus a\check{B}_l \backslash X\oplus a\check{W}_l} w_l(a,z)dz.
\end{align*}

\subsection{Local digital estimators for intrinsic volumes}\label{localal}
We finally specialize the definition of local digital estimators to intrinsic volumes. The definition used in \cite{kampf2,am2} is a special case of this.

Suppose $X\subseteq \R^d$ is a compact convex set. The intrinsic volumes $V_q(X)$ are defined for  $q=0,\dots, d$ to be the coefficients in the well-known Steiner formula 
\begin{equation*}
\Ha^d(X\oplus B(r)) = \sum_{q=0}^d r^{d-q}\kappa_{d-q} V_q(X)
\end{equation*}
for the volume of the Minkowski sum $X\oplus B(r)$ of $X$ and the ball $B(r)\subseteq \R^d $ of radius $r$. Here $\kappa_q$ is the volume of the unit ball in $\R^q$. The intrinsic volumes can be generalized to the class of sets of positive reach, see \cite{federer}.

Each $V_q$ is the total measure of the $q$'th curvature measure $\Phi_q(X;\cdot)$ on $\R^d$, see \cite{schneider}. Thus
\begin{equation*}
V_q(X)= n^{-d}\sum_{z\in \La} \Phi_q(X;aC_{z}^n).
\end{equation*}
This justifies the use of a local algorithm $\hat{V}_q$ for estimating ${V}_q(X)$, i.e. an algorithm of the form 
\begin{align*}%\label{indic}
\hat{V}_q^{a\La_c} (X) ={}& \sum_{z \in \La } \sum_{l=0}^{2^{n^d}-1}  w_l^{(q)}(a,a(z+c)) \\
& \times \mathds{1}_{  X\ominus a\check{B}_l \backslash X\oplus a\check{W}_l}(a(z+c))
\end{align*}
where $w_l^{(q)}(a,a(z+c))$ can be thought of as an estimate for $n^{-d^2}\Phi_q(X;a(C_{z}^n+c))$.  As $\Phi_q(X;\cdot)$ is rotation and reflection invariant, Proposition \ref{rotprop} justifies choosing the weights to be rotation and reflection invariant as well. Moreover, $\Phi_q(X;\cdot)$ is translation invariant so it is natural to require the weights to be so too, i.e.\ $w_l^{(q)}(a,z)=w_l^{(q)}(a)$.
In order to get finite estimators for compact sets, we always assume that $w_0^{(q)}(a)=0$. %We shall only consider the case $q<d$, in which $\Phi_q(X;\cdot)$ is concentrated on $\partial X$. Thus we will assume $w_{2^{n^d}-1}^{(q)}(a)=0$. %This will be justified better in Section \ref{multi}. 

We thus arrive at the following definition of a local digital estimator for $V_q$:
\begin{definition}\label{funddef}
For  $0\leq q \leq d$,
a local digital estimator for $V_q$ is an estimator of the form
\begin{equation}\label{funddefeq}
\hat{V}_q^{ a\La_c}(X)= \sum_{l=1}^{2^{n^d}-1} w_l(a) N_l(X\cap a\La_c)
\end{equation} 
where 
\begin{equation*}
N_l(X\cap a\La_c)= \sum_{z \in \La }  \mathds{1}_{  X\ominus a\check{B}_l \backslash X\oplus a\check{W}_l}(a(z+c))
\end{equation*} 
is the total number of occurrences of the configuration $(B_l,W_l)$ in the image $X\cap a\La_c$. The weights are assumed to be motion and reflection invariant. 
\end{definition}
Throughout this paper, a local digital estimator for $V_q$ will mean an estimator of the form \eqref{funddefeq}.
We often skip the superscripts $a\La_c$ and $(q)$ in the notation for the estimator and the weights and write $\hat{V}_q(X)$ and $w_l(a)$, respectively. 

In applications, the weights are usually chosen to be homogeneous of degree $q$: $w_l^{(q)}(a)=a^qw_l^{(q)}$ for some constants $w_l^{(q)}\in \R$, motivated by the homogeneity pro\-per\-ty: 
\begin{equation*}
\Phi_q(aX;aA)=a^q\Phi_q(X;A).
\end{equation*}
 However, in \cite{kampf2}, also the case of general functions is considered. 
In this paper, we will not assume homogeneity unless explicitly specified.

If an algorithm is not asymptotically unbiased, the worst case relative asymptotic bias measures the bias: 
\begin{definition}
The worst case relative asymptotic bi\-as of an estimator $\hat{V}_q$ for $V_q$ on a class $\mathcal{S}$ of compact convex sets with non-empty interior is given by
\begin{equation*}
\sup_{X\in \mathcal S} \frac{\vert\lim_{a\to 0}E\hat{V}_q^{a\La_c}(X)-V_q(X)\vert}{V_q(X)}.
\end{equation*}
\end{definition}
As long as we restrict ourselves to $\mathcal{S}$, this agrees with the definition in~\cite{kampf2}.
By Proposition \ref{rotprop}, the worst case relative asymptotic bias is minimized by an algorithm with rotation and reflection invariant weights. 

%In \cite{klette}, the following conjecture is formulated in 3D and in \cite{digital} in 2D:
%\begin{conj}
%No local algorithm for estimating $V_{d-1}$ is multigrid convergent on the class of convex sets.
%\end{conj}
%The definition of local algorithm used by these authors is slightly different from ours, as it is based on a reconstruction of the boundary of the underlying set as a union of polygons that is determined only by the $n\times \dotsm \times n$ configurations. Then $V_{d-1}$ is estimated as the sum of the areas of each polygon or, more generally, as a function of each individual polygon. In the latter case, the corresponding conjecture was shown in 2D for straight line segments in \cite{tajine}.
%
%In this paper, we shall prove the following stronger statement:
%\begin{thm}
%No local algorithm defines an asymptotically unbiased (and hence multigrid convergent) estimator for $V_q$, $q<d$, on the class of convex full-dimensional polytopes in $\R^d$.
%\end{thm}

\section{Local estimators for the intrinsic volumes of polytopes}\label{polytopes}
We first consider local digital estimators for intrinsic volumes on the class $\mathcal{P}^d$ of compact convex polytopes in $\R^d$ with non-empty interior. 

We will use the following notation: for a set $A\subseteq \R^d$, we denote by $\aff(A)\subseteq \R^d$ the smallest affine linear subspace containing $A$ and by $\lin(A)\subseteq \R^d$ the linear subspace parallel to $\aff(A)$. For a set of vectors $u_1,\dots , u_N$, we denote by $\pos(u_1,\dots , u_N)$ the set of linear combinations of $u_1,\dots , u_N$ with non-negative coefficients.

\subsection{The space of polytopes}\label{polspace}
The set $\mathcal{P}^d$ is usually given the topology induced by the Hausdorff metric, see \cite[Section 1.8]{schneider}. As our main Theorem \ref{main1} is stated for almost all polytopes, we need an appropriate measure on the induced Borel $\sigma$-algebra in order to make sense of the statement. However, the choice of such a measure is not unambiguous. The most natural way of describing a polytope is either as the convex hull of its vertex set or as an intersection of halfspaces. The parameters describing the vertices and halfspaces, respectively, can be used to parametrize $\mathcal{P}^d$, but this leads to two very different measures. 
In the first case, almost all polytopes will be simplicial while non-simple polytopes constitute a set of positive measure. In the second case, it is the other way around.
A polytope is called simple if every vertex is the intersection of exactly $d$ facets and it is called simplicial if every facet is a simplex, see e.g.\ \cite{ziegler}.

As we shall be viewing polytopes as intersections of halfspaces, we take the second approach. 
 There may still be different ways of defining a measure, and the best choice depends on the application one has in mind. The one we choose could be relevant in situations where the particles under study arise from random sections of some material. However, the main purpose here is to convince the reader that counter examples to multigrid convergence are plenty on $\mathcal{P}^d$. As Theorem \ref{main1} only claims something to be a zero-set, the theorem will also hold for any measure absolutely continuous with respect the one introduced below. 

A convex polytope can always be written in the form 
\begin{equation}\label{intersect}
P=\bigcap_{i=1}^N H_{u_i,t_i}^-
\end{equation}
where $t_i\in \R$ and $u_i\in S^{d-1}$. The idea is to use the parameters $t_i,u_i$ to parametrize polytopes by.
We denote by $S^{d,N} \subseteq (S^{d-1})^N$ the open subset consisting of $N$-tuples of pairwise different vectors in $S^{d-1}$. A point will be written either as a vector $(u_1,\dots,u_N)$ or as an $N\times d$-matrix $U$. Then \eqref{intersect} is the solution set to the matrix inequality $Ux \leq t$. 

First note that \eqref{intersect} is unbounded if and only if the inequality $Ux\leq 0$ has a non-trivial solution $x$ and \eqref{intersect} is non-empty. The set where $Ux\leq 0$ has a non-trivial solution is closed in $S^{d,N}$. Let $S_c^{d,N} \subseteq (S^{d-1})^N$ denote the complement. Then $S^{d,N}\cap S_c^{d,N}$ is open in $(S^{d-1})^N$.

Next observe that \eqref{intersect} has non-empty interior exactly if there exists a solution $x$ to $Ux<t$. This happens for $(U,t)$ in an open subset 
\begin{equation*}
\mathcal{U}^{d,\leq N} \subseteq (S^{d,N}\cap S^{d,N}_c) \times \R^N. 
\end{equation*}

A point $(U,t) \in \mathcal{U}^{d,\leq N}$ defines a polytope with exactly $N$ facets if and only if for every $i = 1, \dots,N$ there is a solution to $\tilde{U}^ix < \tilde{t}^{i}$ where $\tilde{U}^i$ and $\tilde{t}^i$ are $U$ and $t$ except the $i$th row and the $i$th coordinate have changed sign, respectively. This is again an open subset $\mathcal{U}^{d,N}\subseteq \mathcal{U}^{d,\leq N}$.

Let $\mathcal{P}^{d,N} \subseteq \mathcal{P}^d$ be the subset consisting of polytopes with exactly $N $ facets. Then $\mathcal{P}^d$ is the disjoint union of the subsets $\mathcal{P}^{d,N}$.

There is a surjective map 
\begin{equation*}
P: \mathcal{U}^{d,N} \to \mathcal{P}^{d,N}
\end{equation*}
given by \eqref{intersect}.
This is continuous with respect to the Hausdorff metric on $\mathcal{P}^{d,N}$, as one can see e.g.\ by using \cite[Theorem 1.8.7]{schneider}. If $\Sigma_N$ is the $N$'th sym\-me\-tric group acting on $\mathcal{U}^{d,N}$ by permutation of the pairs $(u_i,t_i)$, then $P$ is the quotient map.

\begin{definition}
Denote by $\nu$ the measure on $\mathcal{P}^{d}$ whose restriction to $\mathcal{P}^{d,N}$ is  $\Ha^{dN}\circ P^{-1}$.
\end{definition}

We introduce the following notation for $Q\in \mathcal{P}^d$: $\mathcal{F}_k(Q)$ denotes the set of $k$-faces of $Q$. The facet with normal vector $u_i$ is denoted by $F_i$. If $Q$ is simple, every $F\in \mathcal{F}_k(Q)$ is the intersection of exactly $d-k$ facets. See e.g.\ \cite{ziegler} for details on the combinatorics of simple polytopes. We index the facets containing $F$ by
\begin{equation*}
I_1(F)=\{i_1^F,\dots,i_{d-k}^F\}\subseteq \{1,\dots , N\},
\end{equation*}
i.e.\ $F=\bigcap_{i\in I_1(F)}F_i$. The ordering is not important here. Let
\begin{equation*}
I_2(F)=\{i\in \{1,\dots, N\}\backslash I_1(F)\mid F_i\cap F \neq \emptyset \}
\end{equation*}
index the facets intersecting $F$ in a lower dimensional face. If $Q$ is simple, this lower dimensional face must have dimension $k-1$.

Let $\mathcal{US}^{d,N}$ denote the set 
\begin{equation*}
\mathcal{US}^{d,N}= \{(U,t)\in \mathcal{U}^{d,N} \mid P(U,t) \text{ is simple}\}
\end{equation*}
and let $\mathcal{US}^{d,N}_\mu$, $\mu \in M$, denote the connected components of $\mathcal{US}^{d,N}$.

\begin{proposition}
\InsertTheoremBreak
\begin{itemize}
\item[(i)] For $I \subseteq \{1,\dots ,N\}$, the set
\begin{align*}
G_{I}=\{ {}& (U,t) \in \mathcal{U}^{d,N} \mid  \\ & \exists x \in \R^d:  \forall i\in I : \langle x ,u_{i}\rangle = t_{i},\, Ux\leq t \}
\end{align*}
is relatively closed in $\mathcal{U}^{d,N}$.
\item[(ii)]
$\mathcal{U}^{d,N}\backslash \mathcal{US}^{d,N}$ is relatively closed in $\mathcal{U}^{d,N}$ and has $\Ha^{dN}$-measure 0.
\item[(iii)] For any $I \subseteq \{1,\dots ,N\}$ of cardinality $|I|=d-k$, $P(U,t)$ has a $k$-face $F$ with $I_1(F)=I$ for either no or all $(U,t)\in \mathcal{US}^{d,N}_\mu$.
\end{itemize} 
\end{proposition}

\begin{proof}
(i) To see this, take a sequence $(U^k,t^k)\in G_{i_1,\dots,i_{s}}$ such that $(U^k,t^k) \to (U,t)$ inside $\mathcal{U}^{d,N}$. Then there is a sequence $x^k$ with $\langle u_{i_j}^k,x^k \rangle = t_{i_j}^k$ and $ U^k x^k\leq t^k$. If the $x^k$ are bounded, there is a convergent subsequence $x^{k_n} \to x$ and it follows by continuity that $\langle u_{i_j},x\rangle = t_{i_j}$ and $ Ux\leq t$. If $x^k$ is unbounded, choose a subsequence such that $|x^{k_n}|\to \infty$ and $\frac{x^{k_n}}{|x^{k_n}|}$ converges to $x\in S^{d-1}$. Then $U^{k_n}\frac{x^{k_n}}{|x^{k_n}|}\leq \frac{t^{k_n}}{|x^{k_n}|}$ and thus in the limit $Ux\leq 0 $, contradicting $U\in S_c^{d,N}$.

(ii) If $P(U,t)$ is not simple, it has a vertex $v$ solving $d+1$ of the equations $\langle u_{i_j},v\rangle = t_{i_j}$, $j=1,\dots , d+1$.
The claim now follows from (i) and the fact that 
\begin{align*}
G_{\{i_1,\dots,i_{d+1}\}}\subseteq \{ {}&(U,t) \in \mathcal{U}^{d,N} \mid  \exists x \in \R^d: \\ & \forall j=1,\dots,d+1 : \langle u_{i_j},x\rangle = t_{i_j}\},
\end{align*}
since the latter has $\Ha^{dN}$-measure 0.

(iii)
First assume $k=0$. By the definition of simple polytopes, the set of $(U,t)\in \mathcal{US}^{d,N}_\mu$ having a vertex $v$ with $I_1(v)=I$ is $G_{I} \cap \mathcal{US}^{d,N}_\mu $. This is closed by (i). On the other hand, 
\begin{equation}\label{sys}
\langle u_{i},v \rangle=t_{i} \text{ for } i \in I \text{ and }\langle u_{i},v \rangle < t_{i} \text{ for } i\notin I.
\end{equation} 
Uniqueness of $v$ shows that the system of linear equations $\langle u_{i},v \rangle=t_{i} $ for $ i \in I$ has a unique solution in a neighborhood of $(U,t)$, yielding a solution to \eqref{sys} and thus showing that $G_I\cap \mathcal{US}^{d,N}_\mu $ is also open. Hence $G_I\cap \mathcal{US}^{d,N}_\mu \in \{ \mathcal{US}^{d,N}_\mu ,\emptyset \}$, proving the $k=0$ case.

Given $I$ with $|I|=d-k$, 
\begin{equation*}
F= \bigcap_{i\in I} F_i \in \mathcal{F}_k(P(U,t))\cup \{\emptyset \}
\end{equation*}
whenever $(U,t)\in \mathcal{US}^{d,N}_\mu$. If there is a $(U,t)\in \mathcal{US}^{d,N}_\mu$ and a $v\in \mathcal{F}_0(P(U,t))$ with $I \subseteq I_1(v)$, the $k=0$ case shows that $\bigcap_{i\in I_1(v)}F_i \in \mathcal{F}_0(P(U,t))$ must hold for all $(U,t)\in \mathcal{US}^{d,N}_\mu$ and hence, in particular, $ \bigcap_{i\in I} F_i \neq \emptyset$ for all $(U,t) \in \mathcal{US}^{d,N}_\mu$. If there is no $v\in \mathcal{F}_0(P(U,t))$ with $I\subseteq I_1(v)$, $F$ can have no vertices and is hence empty.\qed
\end{proof}
%
%Let $\mathcal{US}^{d,N}$ denote the set 
%\begin{equation*}
%\mathcal{US}^{d,N}=\mathcal{U}^{d,N} \backslash \bigcup_{1\leq i_1<\dots<i_{d+1}\leq N}  G_{i_1,\dots,i_{d+1}}
%\end{equation*}
%corresponding to the simple polytopes under $P$. Again, this is open in $(S^{d-1})^N\times \R^N$. 
%Let $\mathcal{US}^{d,N}_\mu$, $\mu \in M$, denote the connected components of $\mathcal{US}^{d,N}$.

%
%Let $(U,t)\in \mathcal{US}^{d,N}_\mu$.
%A vertex of $P(U,t)$ is a solution $x$ to $\langle u_{i_j},x \rangle=t_{i_j}$ for $j=1,\dots , d$ and $\langle u_{i},x \rangle < t_{i}$ for $i\neq i_j$. On the one hand, $G_{i_1,\dots,i_d}$ is closed in $\mathcal{US}^{d,N}_\mu$. On the other hand, since the solution is unique, the system $\langle u_{i_j},x \rangle=t_{i_j}$ for $j=1,\dots , d$ can be inverted continuously in a neighborhood of $(U,t)$. Thus the system $\langle u_{i_j},x \rangle=t_{i_j}$ for $j=1,\dots , d$ and $\langle u_{i},x \rangle < t_{i}$ for $i\neq i_j$ has a unique solution for either no or all $(U,t)\in \mathcal{US}^{d,N}_\mu$.

%Hence there is a collection of $I_1(v) \subseteq \{1,\dots ,N\}$ with $|I_1(v)|=d$ indexed by $v\in \mathcal{F}_0(\mu)$, such that $\mathcal{US}^{d,N}_\mu \subseteq G_{i_1,\dots,i_d}$ exactly if $\{i_1,\dots,i_d\}=I_1(v)$ for some $v\in \mathcal{F}_0(\mu)$. For $(U,t)\in \mathcal{US}^{d,N}_\mu$ this naturally defines a bijection $\mathcal{F}_0(\mu)\to \mathcal{F}_0(P(U,t))$. Since the face $F=\bigcap F_{i_s}$ is the convex hull of its faces 

The proposition shows that all $P\in P(\mathcal{US}^{d,N}_\mu)$ have the same combinatorial structure.
A path $(U(s),t(s))$ in $ \mathcal{US}^{d,N}_\mu$ defines a path of vertex sets $\mathcal{F}_0(P(U(s),t(s)))$ by the $k=0$ case in the proof of (iii) and continuity of matrix inversion. This can be extended to an isotopy of $P(U(s),t(s))$ by piecewise linearity using a triangulation with vertices in $\mathcal{F}_0(P(U(s),t(s)))$. This restricts to an isotopy of the combinatorially equivalent lower dimensional faces. We therefore speak of the images $P(\mathcal{US}^{d,N}_\mu)=\mathcal{P}^{d,N}_\mu \subseteq \mathcal{P}_d$ as the combinatorial isotopy classes. 

\subsection{Hit-or-miss transforms of polytopes}\label{polHM}
In order to study the asymptotic bias of a local digital estimator $\hat{V}_q$ applied to $P\in \mathcal{P}^d$, we must consider
\begin{equation*}
E\hat{V}_q({P})=\sum_{l=1}^{2^{n^d}-1} w_l(a) EN_l({P}\cap a\La_c).
\end{equation*}
By \eqref{intHM}, 
\begin{equation}\label{E=HM}
EN_l({P} \cap a\La_c) = a^{-d} \det(\La)^{-1} \Ha^{d}( P\ominus a\check{B}_l\backslash P\oplus a\check{W}_l).
\end{equation}
Thus, we need to describe the volume of hit-or-miss transforms of polytopes.

Suppose $P\in \mathcal{P}^{d,N}$ is given by
\begin{equation*}
P(U,t) = \bigcap_{i=1}^N H^-_{u_i,t_i}.
\end{equation*}
Let $X_{i,l}$ denote the set
\begin{align*}
X_{i,l}&= (H^-_{u_i,t_i}\ominus a\check{B}_l) \backslash (H^-_{u_i,t_i} \oplus a\check{W}_l)\\ &=H^-_{u_i,t_i-ah(B_l,u_i)}\backslash H^-_{u_i,t_i+ah(\check{W}_l,u_i)}
\end{align*}
for $l = 1,\dots , 2^{n^d}-2$ and 
\begin{align*}
&X_{i,0}= \R^d \backslash H^-_{u_i,t_i+ah(\check{C}_{0,0}^n,u_i)},\\
&X_{i,2^{n^d}-1}= H^-_{u_i,t_i-ah(C_{0,0}^n,u_i)}.
\end{align*}

Then $\R^d$ is the disjoint union of the sets $X_{i,l}$ for $l = 0,\dots , 2^{n^d}-1$. Hence it is also the disjoint union of the sets
\begin{equation*}
X_{l_1,\dots ,l_N} = \bigcap_{i=1}^N X_{i,l_i}
\end{equation*}
for $l_1,\dots,l_N\in \{0,\dots,2^{n^d}-1\}$. 

We also use the multi index notation $X_L= X_{l_1,\dots,l_N}$ for  $L \in \mathcal{L}:=\{1,\dots,2^{n^d}-1\}^N$. We associate to an index $L\in \mathcal{L}$ the index sets $I^L=\{i\mid l_i \neq 2^{n^d}-1\}$ and $J^L= \{i\mid l_i = 2^{n^d}-1\}$. Moreover, we associate the face of $P$ given by $F_L=\bigcap_{i\in I^L} F_i$. If $P $ is simple, this is either $|I^L|$-dimensional or the empty face.

\begin{lemma}\label{polvol}
%Let  $(U,t)\in \mathcal{US}_\mu^{d,N}$. 
Let $\mathcal{W}^s\subseteq (S^{d-1})^s$ be the open subset consisting of linearly independent $s$-tuples of unit vectors.
There are functions $a_j:\mathcal{W}^s \to \R$ for all $j,s =1,\dots ,d$, such that:
\begin{itemize}
\item[(i)] $a_{s}(u_1, \dots , u_{s})>0$ and $a_j(u_1,\dots , u_s)=0$ for $s<j$.
\item[(ii)] Each $a_j$ is rotation invariant and depends analytically on $u_1,\dots , u_s$.
\item[(iii)] If $u_s$ is orthogonal to all $u_i$ with $i<s$,
\begin{equation*}
a_{j}(u_1,\dots,u_s) =\begin{cases} 1 & \text{for }$j=s$,\\
0 & \text{otherwise}.
\end{cases}
\end{equation*}
\item[(iv)]
If $S\subseteq \{1,\dots,s\}$ and $\lin(u_i,i\in S)$ is orthogonal to $\lin(u_i,i\notin S)$, then
$a_{j}(u_1,\dots,u_s) = 0 $
if $s\in S$ and $j\notin S$.

\item[(v)] For for $(U,t)\in \mathcal{US}_\mu^{d,N}$ and $F \in \mathcal{F}_q(P(U,t))$, 
\begin{align}\nonumber
\Ha^q(F)={}&\frac{1}{q!}\sum_{v\in \mathcal{F}_0(F)} \sum_{\sigma \in \Sigma_q} \sum_{j_{d-q+1},\dots , j_d=1}^d
 \prod_{s=d-q+1}^d \\\nonumber
  & a_{j_s}(u_{i_1^F},\dots,u_{i_{d-q}^F}, u_{i_{\sigma(d-q+1)}^v},\dots ,u_{i_{\sigma(s)}^v})\\\label{Haface} 
%&\times b_{j_s}({i_1^F},\dots,{i_{d-q}^F},  {i_{\sigma(d-q+1)}^v},\dots ,{i_{\sigma(s)}^v}).\label{Haface}
&\times t_{{i_1^F},\dots,{i_{d-q}^F},  {i_{\sigma(d-q+1)}^v},\dots ,{i_{\sigma(s)}^v}}(j_s)
\end{align}
where $t_{{i_1},\dots , {i_s}}(j)=t_{i_j}$ and indices are chosen so that $I_1(v)= I_1(F)\cup \{ i_{d-q+1}^v,\dots ,i_{d}^v\}$. 
\end{itemize}
 
% such that for

In particular, the volume of $P(U,t)$ is given by a polynomial in $t_1,\dots, t_N$ with coefficients depending only on $U$:
\begin{equation}\label{volofP}
\Ha^d(P(U,t)) = \frac{1}{d!}\sum_{k_1,\dots , k_d= 1}^N a_{k_1,\dots,k_d}(U) \prod_{s=1}^d t_{k_s}
\end{equation}
where
\begin{align}
\label{ajformel}
a_{k_1,\dots,k_d}(U) = &\sum_{v\in \mathcal{F}_0(\bigcap_{m=1}^d F_{k_m})} \sum_{\sigma \in \Sigma_d } \sum_{j_{1},\dots , j_d=1}^d
\\
&\prod_{s=1}^d a_{j_s}( u_{i_{\sigma(1)}^v},\dots ,u_{i_{\sigma(s)}^v}) \mathds{1}_{\{i_{\sigma(j_s)}^v=k_s\}}.
\nonumber
\end{align}
%
%In fact, there are functions $a_j(u_1, \dots , u_s)$ for all $j,s =1,\dots ,d$, defined whenever $(u_1,\dots ,u_s)\in (S^{d-1})^s$ are linearly independent, such that for
%$F \in \mathcal{F}_q(P(U,t))$, 
%\begin{align}\nonumber
%\Ha^q(F)={}&\frac{1}{q!}\sum_{v\in \mathcal{F}_0(F)} \sum_{\sigma \in \Sigma_q} \sum_{j_{d-q+1},\dots , j_d=1}^d
% \prod_{s=d-q+1}^d \\\label{Haface}  & a_{j_s}(u_{i_1^F},\dots,u_{i_{d-q}^F}, u_{i_{\sigma(d-q+1)}^v},\dots ,u_{i_{\sigma(s)}^v})\\
%%&\times b_{j_s}({i_1^F},\dots,{i_{d-q}^F},  {i_{\sigma(d-q+1)}^v},\dots ,{i_{\sigma(s)}^v}).\label{Haface}
%&\times t_{{i_1^F},\dots,{i_{d-q}^F},  {i_{\sigma(d-q+1)}^v},\dots ,{i_{\sigma(s)}^v}}(j_s)\nonumber
%\end{align}
%where $t_{{i_1},\dots , {i_s}}(j)=t_{i_j}$ and indices are chosen so that $I_1(v)= I_1(F)\cup \{ i_{d-q+1}^v,\dots ,i_{s}^v\}$. 
\end{lemma}

We sometimes write $a_j(u_{i_1},\dots,u_{i_s})=a_{i_1,\dots,i_s}(j)$ to keep notation short.

The existence of the formula \eqref{volofP} is basically \cite[Lemma 5.1.2]{schneider}. The remaining claims essentially follow by writing out the details of the proof of that lemma. 

\begin{proof}
For $j\leq q$,  the normalized projection of $u_s$ onto the subspace $\lin(u_1,\dots, u_{s-1})^\perp$ is given by unique linear combination
\begin{equation}\label{normproj}
\sum_{j=1}^s a_j(u_1,\dots , u_s)u_j
\end{equation} 
by the Gram-Schmidt formula. This defines the functions $a_j(u_1,\dots, u_s)$.
We set $a_j(u_1,\dots, u_s)=0$ for $j>s$.
The functions $a_j$ clearly satisfy (i)--(iv) by the Gram-Schmidt formula.

To prove (v), we use the identity
\begin{equation*}
\Ha^{d}(P(U',t'))= \frac{1}{d}\sum_{i=1}^N h(P(U',t'),u_i')\Ha^{d-1}(F_i'),
\end{equation*}
 see \cite[Lemma 5.1.2]{schneider}, which holds for any polytope $P(U',t')$. We apply this inductively to the $q$-faces of $P(U,t)$. The identity \eqref{Haface} clearly holds for $q=0$, the empty product being equal to 1.

Let $F \in \mathcal{F}_q(P)$ be given and let $F^{\prime}\in \mathcal{F}_{q-1}(P)$ be a face of $F$ with $I_1(F')=I_1(F)\cup \{i_{d-q+1}^{F^{\prime}}\}$. The normal vector $u(F,F^{\prime})$ of $F$ at $F^{\prime}$ is exactly the normalized projection of $u_{i_{d-q+1}^{F'}}$ onto $\lin(u_{i_1^{F}},\dots, u_{i_{d-q}^{F}})^\perp$ given by~\eqref{normproj}.

Since $ F' \subseteq \bigcap_{i\in I_1(F)} \partial H_{u_i,t_i}^-$, it follows that
\begin{align*}
h(F,u(F,{}& F^{\prime}))\\
={}& h\Big(F', \sum_{j=1}^{d-q+1} a_j(u_{i_1^{F}},\dots ,u_{i_{d-q}^{F}}, u_{i_{d-q+1}^{F^{\prime}}})u_{i_j^{F^\prime}}\Big)\\
={}&\sum_{j=1}^{d-q+1} a_j(u_{i_1^{F}},\dots ,u_{i_{d-q}^{F}}, u_{i_{d-q+1}^{F^{\prime}}})\\
&\times t_{{i_1^{F}},\dots ,{i_{d-q}^{F}}, {i_{d-q+1}^{F^{\prime}}}}(j).
\end{align*}
Thus by induction,
\begin{align*}
\Ha^{q}{}&(F)= \frac{1}{q}\sum_{F^{\prime}\in \mathcal{F}_{q-1}(F)} \sum_{j=1}^{d-q+1} a_{{i_1^{F}},\dots ,{i_{d-q}^{F}}, {i_{d-q+1}^{F^{\prime}}}}(j)\\
& \times t_{{i_1^{F}},\dots ,{i_{d-q}^{F}}, {i_{d-q+1}^{F^{\prime}}}}(j)
\\ 
&\times \frac{1}{(q-1)!}\sum_{v\in \mathcal{F}_0(F^{\prime})} \sum_{\sigma \in \Sigma_{q-1}} \sum_{j_{d-q+2},\dots, j_d=1}^d \\
&\times \prod_{s=d-q+2}^d  a_{{i_1^{F}},\dots ,{i_{d-q}^{F}}, {i_{d-q+1}^{F^{\prime}}},{i_{\sigma(d-q+2)}^v},\dots ,{i_{\sigma(s)}^v}}(j_s)\\
&\times t_{{i_1^{F}},\dots ,{i_{d-q}^{F}}, {i_{d-q+1}^{F^{\prime}}},{i_{\sigma(d-q+2)}^v},\dots ,{i_{\sigma(s)}^v}}(j_s)\\
={}& \frac{1}{q!} \sum_{F^{\prime}\in \mathcal{F}_{q-1}(F)} \sum_{v\in \mathcal{F}_0(F^{\prime})} \sum_{\sigma \in \Sigma_{q-1}} \sum_{j_{d-q+1},\dots, j_d=1}^d\\
 &a_{{i_1^{F}},\dots ,{i_{d-q}^{F}}, {i_{d-q+1}^{F^{\prime}}}}({j_{d-q+1}})
 t_{{i_1^{F}},\dots ,{i_{d-q}^{F}}, {i_{d-q+1}^{F^{\prime}}}}(j_{d-q+1})\\
&\times \prod_{s=d-q+2}^d a_{{i_1^{F}},\dots ,{i_{d-q}^{F}}, {i_{d-q+1}^{F^{\prime}}},{i_{\sigma(d-q+2)}^v},\dots ,{i_{\sigma(s)}^v}}({j_{s}})\\
& \times t_{{i_1^{F}},\dots ,{i_{d-q}^{F}}, {i_{d-q+1}^{F^{\prime}}},{i_{\sigma(d-q+2)}^v},\dots ,{i_{\sigma(s)}^v}}(j_s)\\
={}&\frac{1}{q!}\sum_{v\in \mathcal{F}_0(F)} \sum_{\sigma \in \Sigma_{q}} \sum_{j_{d-q+1},\dots j_d=1}^d\\ &\prod_{s=d-q+1}^d a_{{i_1^{F}},\dots ,{i_{d-q}^{F}}, {i_{\sigma(d-q+1)}^v},\dots ,{i_{\sigma(s)}^v}}(j_s) \\
&\times t_{{i_1^{F}},\dots ,{i_{d-q}^{F}}, {i_{\sigma(d-q+1)}^{v}},\dots ,{i_{\sigma(s)}^v}}(j_s).
\end{align*}

The last claim of the lemma follows by taking $q=d$ and observing that 
\begin{equation*}
t_{{i_{\sigma(1)}^{v}},\dots ,{i_{\sigma(s)}^v}}(j_s)=t_{k_s}
\end{equation*}
if and only if $i_{\sigma(j_s)}^v=k_s$.
\qed
\end{proof}

Given a multi index $L\in \mathcal{L}$, we use the notation for $i\in I^L$: 
\begin{align*}
&\beta_i=-h(B_{l_i},u_i),\\
&\omega_i= h(\check{W}_{l_i},u_i),\\
&\zeta_i = -h(C_{0,0}^n,u_i),\\
&\delta_L(U)=\prod_{i\in I^L} \mathds{1}_{\{\beta_i > \omega_i\}}.
\end{align*}
For a given index set $k_1,\dots,k_d$, let 
\begin{equation*}
n(i)=|\{s\in\{1,\dots , N\}:k_s=i\}|.
\end{equation*}
 
\begin{lemma} \label{volpiece}
Let $(U,t)\in \mathcal{US}^{d,N}_\mu$ and $L\in \mathcal{L}$ be given. Then for a sufficiently small, $\Ha^{d}(X_{L})$ is a homogeneous polynomial of degree $d$ in the numbers $(t_i+a\beta_i)$ and $(t_i+a\omega_i)$ for $i\in I_1(F_L)$ and $(t_i+a\zeta_i)$ for $i\in I_2(F_L)$ with coefficients depending only on $U$. In particular, it is a homogeneous polynomial of degree $d$ in $a,t_1,\dots,t_N$ given by
\begin{align}\nonumber
\Ha^{d}(X_{L})={}& \delta_L(U)\frac{1}{d!}\sum_{k_1,\dots , k_d\in I_1(F_L)\cup I_2(F_L)} a_{k_1,\dots,k_d}(U) \\\nonumber 
&\times \prod_{i \in I_1(F_L)}\sum_{s_i=1}^{n(i)}\binom{n(i)}{s_i} a^{s_i}t_{i}^{n(i)-s_i}(\beta_{i}^{s_i}-\omega_{i}^{s_i}) \\\label{paen}
&\times \prod_{j \in I_2(F_L)} (t_{j}+a\zeta_{j})^{n(j)}.
\end{align}
In particular, $\Ha^{d}(X_{L})=0$ if $F_L=\emptyset$.

As a polynomial in $a$, the lowest order term is
\begin{align*}
a^{|I_1(F_L)|}{}&\delta_L(U)\frac{1}{d!}\sum_{k_1,\dots , k_d\in I_1(F_L)\cup I_2(F_L)} a_{k_1,\dots,k_d}(U) \\
 \times {}& \prod_{i \in I_1(F_L)} {n(i)} t_{i}^{n(i)-1}(\beta_{i}-\omega_{i}) \prod_{j \in I_2(F_L)} t_{j}^{n(j)}.
\end{align*} 
\end{lemma}

\begin{proof}
We must compute the volume of
\begin{align*}
X_{L}  ={}& \bigcap_{i\in I^L} (H^-_{u_i,t_i+a\beta_i}\backslash H^-_{u_i,t_i+a\omega_i})\\
&\cap \bigcap_{j\in J^L} H^-_{u_j,t_j+a\zeta_j}. \\
%&= \left(\bigcap_{i\in I_1^L} H^-_{u_i,b_i + a\eps_i} \cap \bigcap_{i\in I_2^L} H^-_{u_i,b_i+a\zeta_i}\right) \backslash  \bigcup_{i\in I_1^L} \left(H^-_{u_i,b_i+a\omega_i} \cap \bigcap_{j\in I_1^L\backslash \{i\}} H^-_{u_j,b_j+a\beta_j} \cap \bigcap_{i\in I_2^L} H^-_{u_i,b_i+a\zeta_i} \right).
\end{align*}
Clearly, if $\delta_{L}(U)=0$, this is empty.
For $I\subseteq I^L$ let 
\begin{align*}
X_I={}&\bigcap_{i\in I}  H^-_{u_i,t_i+a\omega_i} \cap \bigcap_{j\in I^L\backslash I} H^-_{u_j,t_j+a\beta_j} \\
& \cap \bigcap_{k\in J^L} H^-_{u_k,t_k+a\zeta_k} .
\end{align*}
Then $X_I\cap X_J = X_{I\cup J}$ since $\omega_i \leq \beta_i$ for all $i\in I^L$ and
\begin{equation*}
X_{L} = X_{\emptyset} \backslash \bigcup_{i\in I^L} X_{\{i\}}.
%&= \left(\bigcap_{i\in I_1^L} H^-_{u_i,b_i + a\eps_i} \cap \bigcap_{i\in I_2^L} H^-_{u_i,b_i+a\zeta_i}\right) \backslash  \bigcup_{i\in I_1^L} \left(H^-_{u_i,b_i+a\omega_i} \cap \bigcap_{j\in I_1^L\backslash \{i\}} H^-_{u_j,b_j+a\beta_j} \cap \bigcap_{i\in I_2^L} H^-_{u_i,b_i+a\zeta_i} \right).
\end{equation*}
For $a$ sufficiently small, $X_I \in \mathcal{P}^{d,N}_\mu$ for all $I$ by openness of $ \mathcal{US}^{d,N}_\mu$. 
Let $Q(t)=\Ha^{d}(P(U,t))$ be the polynomial in~\eqref{volofP} and write
\begin{equation*}
\xi_i(I) = \omega_i 1_{i\in I} + \beta_i 1_{i \in I^L\backslash I}+\zeta_i 1_{i\in J^L}.
\end{equation*}
Then the inclusion-exclusion principle yields: 
\begin{align*}
\Ha^d{}&(X_{L}) = \sum_{I\subseteq I^L}(-1)^{|I|} \Ha^d( X_I)\\=&
\sum_{I\subseteq I^L}(-1)^{|I|}  Q(t_1 + a\xi_1(I),\dots , t_N + a\xi_N(I)) \\
={}& \frac{1}{d!}\sum_{k_1,\dots , k_d} a_{k_1,\dots,k_d} \sum_{I\subseteq I^L}(-1)^{|I|}  \prod_{i \in I}(t_{i}+a\omega_{i})^{n(i)} \\
&\times \prod_{ j \in I^L\backslash I} (t_{j}+a\beta_{j})^{n(j)} \prod_{k \in J^L} (t_{k}+a\zeta_{k})^{n(k)} \\
={}& \frac{1}{d!}\sum_{k_1,\dots , k_d} a_{k_1,\dots,k_d}\\
& \prod_{i\in I^L}((t_{i}+a\beta_{i})^{n(i)}-(t_{i}+a\omega_{i})^{n(i)} ) \\
&\times \prod_{j \in J^L} (t_{j}+a\zeta_{j})^{n(j)}\\
={}& \frac{1}{d!}\sum_{k_1,\dots , k_d\in I_1(F_L)\cup I_2(F_L)} a_{k_1,\dots,k_d}\\
&\times \prod_{i \in I_1(F_L)}
\sum_{s_i=1}^{n(i)}\binom{n(i)}{s_i}a^{s_i} t_{i}^{n(i)-s_i}(\beta_{i}^{s_i}-\omega_{i}^{s_i})\\
&\times \prod_{j \in I_2(F_L)} (t_{j}+a\zeta_{j})^{n(j)}.
\end{align*}
%The last equality follows from the fact that $I^L=I_1(F_L)$ and since  only terms with $I_1(F_L)\subseteq \{ j_1,\dots, j_d\}$ contribute, the description of $a_{j_1,\dots,j_d}$ in Lemma \ref{polvol} shows that $a_{j_1,\dots,j_d}=0$ unless $\{ j_1,\dots, j_d\}\subseteq I_1(F_L)\cup I_2(F_L)$.\qed
The last equality follows from the fact that $I^L=I_1(F_L)$ and that $I_1(F_L)\subsetneq \{ k_1,\dots, k_d\}$ implies $n(i)=0$ for some $i\in I_1(F_L)$ so that the product over $i\in I_1(F_L)$ is zero. Thus  only terms with $I_1(F_L)\subseteq \{ k_1,\dots, k_d\}$ contribute and the description of $a_{k_1,\dots,k_d}$ in \eqref{ajformel} therefore shows that $a_{k_1,\dots,k_d}=0$ unless 
\begin{equation*}
\{ k_1,\dots, k_d\}\subseteq I_1(F_L)\cup I_2(F_L).
\end{equation*}
\qed
\end{proof}

\subsection{Asymptotic behavior of the estimators}\label{asymptotics}
For $x\in X_{l_1,\dots ,l_N}$, 
\begin{equation*}
(x+aC_{0,0}^n) \cap P = x+ a\bigcap_{i=1}^N B_{l_i}.
\end{equation*}
We denote the configuration $\bigcap_{i=1}^N B_{l_i}$ by $B_{l_1,\dots , l_N}$ and the corresponding weight is denoted by $w_{l_1,\dots , l_N}(a)$ or $w(\bigcap_{i=1}^N B_{l_i},a)$. 
Note that if one of the $l_i$ equals 0, then $B_{l_1,\dots , l_N}=B_0=\emptyset$. For $L\in \mathcal{L}$, we also use the notation $B_L$ and $w_L(a)$.

As explained in the preceding section, $\R^d$ is the disjoint union of the sets $X_L$, $L\in \mathcal{L}$, and we have
\begin{equation*}
X_L\cap (P\ominus a\check{B}_l\backslash P\oplus a\check{W}_l)=\begin{cases}
X_L, & B_L=B_l\\
\emptyset, & B_L\neq B_l\\
\end{cases}
\end{equation*}
Hence \eqref{E=HM} yields:

\begin{corollary}\label{brik}
Let $P\in \mathcal{P}^{d,N}_\mu$ be a polytope. Then for $l\neq 0$,
\begin{equation*}
EN_l(P\cap a\La_c)=  a^{-d} \det(\La)^{-1} \sum_{L\in \mathcal{L}} \Ha^d(X_{L})\mathds{1}_{\{B_{L}=B_l\}}.
\end{equation*}
It follows that
\begin{equation*}
E\hat{V}_q(P) = a^{-d} \det(\La)^{-1} \sum_{L\in \mathcal{L}} w_{L}(a)\Ha^d(X_L).
\end{equation*}
where $\Ha^d(X_L)$ is given by Lemma \ref{volpiece}.
\end{corollary}
 
For a local estimator $\hat{V}_q$, we introduce the following notation:
\begin{align*}
\mathcal{E}^N&=\{P\in \mathcal{P}^{d,N} \mid \lim_{a\to 0} E\hat{V}_q(P) \text{ exists } \},\\
\mathcal{V}^N&=\{P\in \mathcal{E}^N \mid \lim_{a\to 0} E\hat{V}_q(P) = V_q(P) \}.
\end{align*}
Similarly, for a combinatorial isotopy class $\mathcal{P}^{d,N}_\mu$ of simple polytopes, $\mathcal{E}^N_\mu=\mathcal{P}_\mu^{d,N}\cap \mathcal{E}^N$ and $\mathcal{V}^{d,N}_\mu=\mathcal{P}_\mu^{d,N} \cap \mathcal{V}^N$. 

\begin{lemma}\label{estcor}
There exist measurable subsets ${V}^N_\mu, {E}^N_\mu $ of $ (S^{d-1})^N$ satisfying
\begin{align*}
&\tilde{\mathcal{E}}^{ N}_\mu  := (E^N_\mu \times \R^N) \cap \mathcal{US}^N_\mu \subseteq \mathcal{E}^{ N}_\mu, \\
&\tilde{\mathcal{V}}^{N}_\mu  := (V^N_\mu \times \R^N) \cap \mathcal{US}^N_\mu \subseteq \mathcal{V}^{ N}_\mu, \\
&\Ha^{dN } (\mathcal{E}^N_\mu \backslash \tilde{\mathcal{E}}^{N}_\mu)= \Ha^{dN } (\mathcal{V}^N_\mu \backslash \tilde{\mathcal{V}}^{N}_\mu) =0,
\end{align*}
such that on $\tilde{\mathcal{E}}^N_\mu$, $\lim_{a\to 0} E\hat{V}_q(P(U,t))$ is a polynomial in $t_1,\dots,t_N$ with coefficients depending only on $U$ and on $\tilde{\mathcal{V}}^{N}_\mu \subseteq \tilde{\mathcal{E}}^N_\mu$, this is homogeneous of degree $q$. 
%
%Similarly, there is a set ${V}^N_\mu \subseteq (S^{d-1})^N$ such that if $\tilde{\mathcal{V}}^{N}_\mu = (V^N_\mu \times \R^N) \cap \mathcal{US}^N_\mu $, then $\tilde{\mathcal{V}}^{ N}_\mu \subseteq \mathcal{V}^{ N}_\mu $ and $=0$ and  $\lim_{a\to 0} E\hat{V}_q(P)$ is homogeneous of degree $q$ as a polynomial in $b_1,\dots,b_N$ on $\tilde{\mathcal{V}}^{N}_\mu$.
\end{lemma}

\begin{proof}
Let 
\begin{equation*}
E^N_\mu=\{U\in (S^{d-1})^N \mid \Ha^N(\mathcal{E}^N_\mu\cap (\{U\}\times \R^N))>0 \}.
\end{equation*}
 Then 
\begin{align*}
\Ha^{dN}(\mathcal{E}^N_\mu \backslash \tilde{\mathcal{E}}^{ N}_\mu)&=\int_{(S^{d-1})^N \backslash E^N_\mu}\int_{\R^N}1_{\mathcal{E}^N_\mu} d\Ha^N d\Ha^{(d-1)N}\\& =0. 
\end{align*}

By Lemma \ref{volpiece} and  Corollary \ref{brik}, $E\hat{V}_q({P})$ has the form
\begin{equation*}
\sum_{\substack{n_1,\dots , n_N=0,\\ \sum n_i \leq d}}^{d-1} H_{n_1,\dots,n_N}(a) \prod_{i=1}^N t_i^{n_i}.
\end{equation*}
For a fixed $U\in E^N_\mu$, the function $H_{n_1,\dots,n_N}(a)$ depends only on $a$ and the limit when $a\to 0$ exists for all $t_1,\dots ,t_N$ in a set of non-zero $\Ha^N$-measure. It follows from linear independence of the monomials $\prod_{i=1}^N t_i^{n_i}$ that each limit $\lim_{a\to 0} H_{n_1,\dots,n_N}(a)$ must exist.  Denote this limit by $H_{n_1,\dots,n_N}$. Then
\begin{align}
\begin{split}\label{est-bpol}
\lim_{a\to 0}{}& \sum_{\substack{n_1,\dots , n_N=0,\\ \sum n_i \leq d}}^{d-1} H_{n_1,\dots,n_N}(a) \prod_{i=1}^N t_i^{n_i}\\ ={}&\sum_{\substack{n_1,\dots , n_N=0,\\ \sum n_i \leq d}}^{d-1} H_{n_1,\dots,n_N}\prod_{i=1}^N t_i^{n_i}
\end{split}
\end{align}
and in particular, $\tilde{\mathcal{E}}^{N}_\mu \subseteq \mathcal{E}^{ N}_\mu $.

Similarly, define 
\begin{equation*}
V^N_\mu=\{U\in (S^{d-1})^N \mid \Ha^N(\mathcal{V}^N_\mu \cap (\{U\}\times \R^N))>0 \}. 
\end{equation*}
Recall that
\begin{equation}\label{Vqform}
V_q(P) = \sum_{F\in \mathcal{F}_q(P)} \gamma(F,P) \Ha^q(F) 
\end{equation}
where 
\begin{equation}\label{vinkel}
\gamma(F,P)= \frac{\Ha^{d-q-1}(\text{pos}(u_{i_1^F},\dots,u_{i_{d-q}^F})\cap S^{d-1}) }{\Ha^{d-q-1}(S^{d-q-1})}
\end{equation}
is the external angle of $P$ at $F$ and clearly depends only on $U$. By Lemma~\ref{polvol}, each $\Ha^q(F)$ is a homogeneous polynomial of degree $q$ in $t_1,\dots,t_N$. Thus, for $U\in {E}_\mu^N$, either $\Ha^N(\mathcal{V}^N_\mu \cap (\{U\}\times \R^N))=0$ or the coefficients of~\eqref{est-bpol} and~\eqref{Vqform} must agree. In particular, $H_{n_1,\dots,n_N}=0$ unless $\sum n_i =q$.\qed
%\begin{equation*} 
%V_q(P)= \sum_{S\subseteq \{1,\dots,N\}^{q} }H_S \prod_{i\in S} b_i
%\end{equation*}
%for certain coefficients
\end{proof}

Let $\tilde{\mathcal{E}}^{ N} =\bigcup_{\mu \in M} \tilde{\mathcal{E}}^{ N}_\mu$ and $\tilde{\mathcal{V}}^{ N} =\bigcup_{\mu \in M} \tilde{\mathcal{V}}^{ N}_\mu$. 

\begin{corollary}\label{weights}
Given a local estimator $\hat{V}_q$, there is a local estimator $\hat{V}'_q$ with polynomial weights such that on  $\tilde{\mathcal{E}}^{ N}$, $\lim_{a\to 0}E\hat{V}_q(P)=\lim_{a\to 0}E\hat{V}'_q(P)$. Moreover, there is an estimator $\hat{V}''_q$ with homogeneous weights of degree $q$ such that $\lim_{a\to 0}E\hat{V}''_q(P)={V}_q(P)$ on $\tilde{\mathcal{V}}^{N}$.
\end{corollary}

\begin{proof}
By Lemma \ref{volpiece} and Corollary \ref{brik}, $E\hat{V}_q(P)$ takes the form
\begin{equation*}
E\hat{V}_q(P)=\sum_{l=1}^{2^{n^d}-1} w_l(a)\sum_{k=0}^d a^{k-d} c_{l,k}(P)
\end{equation*}
where the coefficients $c_{l,k}(P)\in \R$ have degree $d-k$ in $t$ and depend only on $P\in P(\mathcal{US}^{d,N})$.

For each $k=0,\dots, d$, choose $M_k \subseteq \{1,\dots , 2^{n^d}-1\}$ maximal with no linear relation between the coefficients $c_{l,k}(P)$ with $l\in M_k$ that holds for all $P\in P(\tilde{\mathcal{E}}^{N})$.
In particular, for $l\in M_k$ there are functions 
\begin{equation*}
w_{l,k}(a)= w_l(a) + \sum_{s \notin M_k} \alpha^s_{l,k} w_s(a)
\end{equation*}
for suitable $\alpha_{l,k}^s \in \R$ such that
\begin{equation}\label{limred}
\lim_{a\to 0} E\hat{V}_q(P)=\lim_{a\to 0} \sum_{k=0}^{d} \sum_{l\in M_k} w_{l,k}(a) a^{k-d} c_{l,k}(P)
\end{equation}
for all $P\in P( \tilde{\mathcal{E}}^{N})$.
By the proof of Lemma \ref{estcor}, the limit exists for each $k$ term in the sum on $P(\tilde{\mathcal{E}}^{N})$.

Choose $P_m \in P(\tilde{\mathcal{E}}^{ N}) $ for $m\in M_k$ such that the vectors $(c_{l,k}(P_{m}))_{l\in M_k}$ are linearly independent. The existence of the limit \eqref{limred} for all $P_m$ yields an invertible linear system, and solving this shows that also
\begin{equation*}
w_{l,k} :=\lim_{a\to 0 } w_{l,k}(a)a^{k-d}
\end{equation*}
exists for all $l$. 

Let $W$ be the formal vector space spanned by the functions $w_l(a)$ and let $W^q$ be the subspace spanned by $ \{w_{l,q}(a)\mid l\in M_q\} $.
%We show by induction that it is possible to choose polynomials $\tilde{w}_{l,k}(a)$ of degree at least $d-k$ such that $\lim_{a\to 0}\tilde{w}_{l,k}(a) a^{k-d} = w_{l,k}$ consistently in the sense that it defines a linear map $\text{span}\{W^q ,q=0,\dots , d\} \to \Pol_d$ where $\Pol_d$ is the set of polynomials with $\R$-coefficients of degree at most $d$.
%
%We show by induction that it is possible to choose polynomials $\tilde{w}_{l,k}(a)$ of degree at least $d$ that are zero below degree $d-k$ such that $\lim_{a\to 0}\tilde{w}_{l,k}(a) a^{k-d} = w_{l,k}$ consistently in the sense that the assignment ${w}_{l,k}(a)\mapsto \tilde{w}_{l,k}(a)$ defines a linear map $\text{span}\{W^q ,q=0,\dots , d\} \to \Pol_d$ where $\Pol_d$ is the set of polynomials with $\R$-coefficients of degree at most $d$.
Let $\Pol_d^k$ be the set of polynomials with $\R$-coefficients of degree at most $d$ that are zero below degree $d-k$.

We will show by induction in $k$ that there is a linear map 
\begin{equation*}
\text{span}\{W^q ,q=0,\dots , k \} \to \Pol_d^k
\end{equation*}
taking ${w}_{l,k}(a)$ to $\tilde{w}_{l,k}(a)$ such that 
\begin{equation*}
\lim_{a\to 0}\tilde{w}_{l,k}(a) a^{k-d} = w_{l,k}. 
\end{equation*}

For $k=0$, choose $\tilde{w}_{l,0}(a)= w_{l,0}a^{d} $. Suppose now that we have chosen $\tilde{w}_{l,k}(a)$ for all $k<q$ defining a map $\text{span}\{ W^k ,k<q \}\to \Pol_d^{q-1}$. 

We know $\lim_{a\to 0} w(a)a^{q-d} $ exists for all $w(a)\in W^q$. Choose a maximal set of independent $w_{l_i,k_i}(a)\in  W^q$, $i\in I$, with $k_i<q$ and extend this by $w_1^q,\dots , w_m^q$ to a basis of $W^q$. Then $\lim_{a\to 0}a^{q-d}w_{l_i,k_i}(a)=0$ and $\lim_{a\to 0}a^{q-d}w_{j}^q(a)=w_j^q$. For 
\begin{equation*}
w_{l,q}(a)=\sum_{i\in I} {\alpha_i}w_{l_i,k_i}(a) + \sum_{j=1}^m \beta_j w_j^q(a)
\end{equation*}
define
\begin{equation*}
\tilde{w}_{l,q}(a)= \sum_{i\in I} \alpha_{l,q}^i \tilde{w}_{l_i,k_i}(a) + \sum_{j=1}^m \beta_{l,q}^j w_j^q a^{d-q}.
\end{equation*}
This extends the map $\text{span}\{ W^k ,k<q \}\to \Pol_d^{q-1}$ to $\text{span}\{ W^k, {k\leq q}\} \to \Pol_d^q$, completing the induction step. 

Any linear extension of the resulting map with $k=d$
%$\text{span}\{W^q ,q=0,\dots , d\} \to \Pol_d^d$ 
to a map $W \to \Pol_d^d$ yields a way of choosing the $w_l(a)$.

To prove the second claim, we choose $M_k$ such that $c_{l,k}(P)$ are independent for $P\in \tilde{\mathcal{V}}^{ N}$. Observe that for $P\in \tilde{\mathcal{V}}^{ N}$, \eqref{limred} is homogeneous in $t$ of degree $d-q$ by Lemma \ref{estcor}. In particular, the limit of each term with $k\neq q$ must vanish. Again, inverting a linear system shows that $w_{l,k}=0$ for $k\neq q$. Thus the inductive construction yields an estimator with homogeneous weights.
\qed
\end{proof}

\subsection{Intrinsic volumes of positive degree}\label{ivpd}

We are finally ready to prove Theorem \ref{main1} which we restate as follows:
\begin{theorem}\label{main1'}
Let $\hat{V}_q$ be any local algorithm for $V_q$ for $1\leq q\leq d-1$.

If $d-q$ is odd, $\hat{V}_q$ is asymptotically biased $\nu$-almost everywhere on $\mathcal{E}^{d,N}$. 

For $d-q$ even, $\hat{V}_q$ is asymptotically biased $\nu$-almost everywhere on $\mathcal{E}^N_\mu$ for all combinatorial isotopy classes $\mu \in M$ corresponding to polytopes having a $(d-q)$-face which is combinatorially isotopic to $\bigoplus_{i=1}^{d-q}[0,e_i]$. In particular, $\hat{V}_q$ is asymptotically biased on a set of positive $\nu$-measure.
\end{theorem}

\begin{proof}
Suppose we are given an estimator $\hat{V}_q$.  Fix a combinatorial isotopy class $\mathcal{US}^{d,N}_\mu$. We want to show that $\Ha^{dN}(\mathcal{V}^N_\mu) = 0$. It is enough to show $\Ha^{(d-1)N}({V}^N_\mu)=0$.

By Corollary \ref{weights}, we may assume that the weights are homogeneous of degree $q$. Then
\begin{equation}\label{estcoef}
\lim_{a\to 0} E\hat{V}_q(P) = \det(\La)^{-1} \sum_{L\in \mathcal{L}} w_{L} A^{q}_{L}(U,t)
\end{equation}
where $A_{L}^{q}$ is the coefficient in front of $a^{d-q}$ in Formula~\eqref{paen} for $\Ha^{d}(X_{L})$. We write $w_L'= \det(\La)^{-1} w_L$ to shorten notation. In particular, \eqref{estcoef} is a homogeneous polynomial in $t_1,\dots, t_N$ of degree $q$. On  $\tilde{\mathcal{V}}_\mu^{ N}$, this must equal 
\begin{equation}\label{truecoef}
V_q(P)= \frac{1}{q!} \sum_{F\in \mathcal{F}_q(P)} \gamma(F,P) \Ha^{q}(F).
%&=  \frac{1}{q!} \sum_{F\in \mathcal{F}_q(P)} \gamma(P,F) \sum_{v\in F} \sum_{\sigma \in \Sigma_q}  \sum_{j_{d-q+1},\dots , j_d=1}^d \prod_{s=d-q+1}^d a_{j_s}(u_{i_1^F},\dots,u_{i_{d-q}^F}, u_{i_{\sigma(d-q+1)}^v},\dots ,u_{i_{\sigma(s)}^v})\\\nonumber
%&\times b_{j_s}(u_{i_1^F},\dots,u_{i_{d-q}^F}, u_{i_{\sigma(d-q+1)}^v},\dots ,u_{i_{\sigma(s)}^v}).
\end{equation}

Choose a $(d-q)$-face $F_I=\bigcap_{i\in I} F_i$ with $|I|=q$. We want to compare the coefficients in front of $\prod_{i\in I} t_i$. Denote the coefficient in \eqref{estcoef} by $H_I$ and the one in \eqref{truecoef} by $G_I$. 
Then $H_I$ must equal $G_I$ on $\tilde{\mathcal{V}}^{ N}_\mu$.
Both $H_I$ and $G_I$ depend only on $U\in (S^{d-1})^N$. In order to show that $\Ha^{(d-1)N}(V^N_\mu)=0$, it is enough to show that almost all points in $V^N_\mu$ have a small neighborhood $W\subseteq (S^{d-1})^N$ with $\Ha^{(d-1)N}(W\cap \{ H_I=G_I \})=0$. 

For $c_1\neq c_2 \in C_{0,0}^n$, let $H_{c_1,c_2}$ denote the hyperplane $\{x\in \R^d \mid \langle x, c_1\rangle=\langle x, c_2\rangle\}$. Let 
\begin{equation*}
D=\bigcup_{c_1\neq c_2 \in C_{0,0}^n}H_{c_1,c_2}.
\end{equation*}
Observe that for a set $S\subseteq C_{0,0}^n$ and a connected component $E$ in $S^{d-1}\backslash D$, there is a unique $s\in S$ such that $h(S,u)=\langle s,u \rangle$ for all $u \in E$. Moreover, all the indicator functions $\delta_l$ are constant on $E$.

Since $\Ha^{d-1}(D)=0$, almost all $(u_1,\dots , u_N) \in V^N_\mu$ belong to $ (S^{d-1}\backslash D)^N$. Let such $U \in V^N_\mu \cap (S^{d-1}\backslash D)^N$ be given. Choose a small connected neighborhood $W$ contained in $ \mathcal{US}_\mu^{d,N} \cap ((S^{d-1}\backslash D)^N\times \R^N)$. Then there are vectors $b_{l}^i\in B_{l}\cup \{0\}$, $w_{l}^i\in \check{W}_{l}\cup \{ 0 \}$, and  $c_i \in C_{0,0}^n$ such that 
\begin{align*}
&h(B_{l},u_i)\delta_l(u_i)=\langle b_{l}^i ,u_i \rangle,\\
&h(\check{W}_{l},u_i)\delta_l(u_i)=\langle w_{l}^i,
 u_i \rangle,\\
&h(C_{0,0}^n,u_i)=\langle c_{i},u_i \rangle,
\end{align*}
whenever $(u_1,\dots , u_N) \in W$.
Thus $H_I$ has the form
\begin{align*}
H_I(U)={}& \sum_{L\in \mathcal{L}} w_L '\mathds{1}_{I\subseteq I_1(F_L)\cup I_2(F_L)} \sum_{k_1,\dots , k_d\in I_1(F_L)\cup I_2(F_L)}\\
& a_{k_1,\dots,k_d}(U)  \prod_{i \in I_1(F_L)} (\langle b_{l_i}^i ,u_i \rangle^{e(i)}-\langle w_{l_i} ,u_i \rangle^{e(i)}) \\
&\times d_{k_1,\dots,k_d}  \prod_{j \in I_2(F_L)} \langle c_{j},u_j \rangle^{e(j)}
\end{align*}
on $W$.  Here $d_{k_1,\dots,k_d}$ are certain constants and $e(i)$ are certain exponents with 
\begin{equation*}
\sum_{i\in I_1(F_L)\cup I_2(F_L)} e(i)= d-q.
\end{equation*}
 In particular, $H_I$ is an analytic function, depending only on the $u_i$ with $i\in I \cup I_2(F_I)$.
%This can be extended to an analytical function defined whenever $u_{i_1^v},\dots,u_{i_d^v}$ are linearly independent for every vertex $v\in \mathcal{F}_0(F_I)$ by Lemma \ref{polvol}. 

Similarly, by \eqref{truecoef} and Lemma \ref{polvol}
\begin{align}\nonumber
G_I(U)= {}&\frac{1}{q!}  \sum_{\substack{F\in \mathcal{F}_q(P)\\ F\cap F_I \neq \emptyset}} \gamma(F,P) \sum_{v\in F} \sum_{\sigma \in \Sigma_q}  \sum_{(j_{d-q+1},\dots , j_d)\in J_{v,\sigma}}\\
&\prod_{s=d-q+1}^d a_{{i_1^F},\dots,{i_{d-q}^F}, {i_{\sigma(d-q+1)}^v},\dots ,{i_{\sigma(s)}^v}}(j_s)\label{GI}
\end{align}
where $J_{v,\sigma}$ are certain index sets. Recall also that $u_{i}$ is the normal vector of $P$ at the facet $F_{i}$ and that $F=\bigcap_{k=1}^{d-q}F_{i_k^F}$. Each $\gamma(F,P)$ is an analytic function of $u_{i_1^F},\dots,u_{i_{d-q}^F}$ which is defined whenever $u_{i_1^F},\dots,u_{i_{d-q}^F}$ are linearly independent. This follows from Schl\"{a}fli's formula \cite{sch}, see also \cite{aomoto}, according to which $\gamma(F,P)$ is analytic as a function of the angles between the faces in $\pos(u_{i_1^F},\dots,u_{i_{d-q}^F} )$, and these angles can again be expressed analytically as functions of $u_{i_1^F},\dots,u_{i_{d-q}^F}$. It follows that $G_I$ is analytic on $W$.

The formulas for $H_I$ and $G_I$, initially defined on $W$, naturally extend to analytic functions $\bar{H}_I,\bar{G}_I:W' \to \R$ where $W' \subseteq (S^{d-1})^{|I \cup I_2(F_I)|}$ is the largest connected subset containing $W$ and such that $u_{i_1^v},\dots,u_{i_d^v}$ are li\-ne\-ar\-ly independent for every  $v\in \mathcal{F}_0(F_I)$.  

Choose a path through independent unit vectors inside $\lin(u_i,i\in I)^{|I|}$ from $(u_i)_{i\in I}$ to an orthonormal frame $(u_i')_{i\in I}$.
Next, for each $u_j$ with $j\in I_2(F)$, choose a path inside $\lin(u_j, u_i,i\in I) \backslash \lin(u_i,i\in I)$ from $u_j$ to its normalized projection onto $\lin(u_i,i\in I)^\perp$ denoted by $u_j'$. Together, this defines a path inside $W'$ from $(u_i)_{i\in I\cup I_2(F_I)}$ to $(u_i')_{i\in I\cup I_2(F_I)}$ such that the $u_{i}'$ with $i\in I$ are orthogonal and each $u_j'$ with $j\in I_2(F_I)$ is orthogonal to all $u_{i}'$ with $i\in I$. 

The proof of Lemma \ref{volpiece} shows that a term with index ${j_1,\dots,j_d}$ in the summation formula for $\Ha^{d}(X_{L})$ can only contribute a $\prod_{i\in I} t_i$ term if $I\subseteq \{k_1,\dots , k_d\}$, and by Lemma \ref{polvol} (iii) and (iv), if ${i_s^v}\in I$, then 
\begin{equation*}
a_{j_s}(u_{i_1^v}',\dots,u_{i_s^v}')=\begin{cases}1,& \text{for }j_s=s,\\
0,& \text{otherwise}.
\end{cases}
\end{equation*}
Moreover, if ${i_s^v} \notin I$, and $j_s\in I$, 
\begin{equation*}
a_{j_s}(u_{i_1^v}',\dots,u_{i_s^v}')=0. 
\end{equation*}
Thus, $a_{k_1,\dots,k_d}$ can only be non-zero if every element of $I$ appears exactly once in $k_1,\dots,k_d$. 
In the formula for $\Ha^{d}(X_{L})$ given in Lemma~\ref{volpiece}, this means that $n(i)=1$ for all $i\in I$. Hence the term $\prod_{i\in I} t_i$ can only appear if $I \subseteq I_2(F_{L})$. Define the index set
\begin{align*}
J_L=\{{}&(k_1,\dots , k_d) \mid  \exists v \in \mathcal{F}_0(F_I\cap F_{L}):\\
& k_1,\dots , k_d\in I_1(v),\, 
\forall j \in I: n(j)=1 \}.
\end{align*}
In the formula for $\Ha^{d}(X_{L})$, the coefficient in front of $\prod_{i\in I} t_i$  applied to the point $(u_i')_{i\in I\cup I_2(F_I)}$ has the form 
\begin{align*}
\frac{1}{d!} {}&\sum_{ (k_1,\dots , k_d)\in J_L } a_{k_1,\dots,k_d} \\
&\times \prod_{i \in I_1(F_L)} (\langle b_{l_i}^i,u_i'\rangle^{n(i)}-\langle w_{l_i}^i,u_i'\rangle^{n(i)})\\
&\times \prod_{j \in I_2(F_L) \backslash I} \langle c_{j},u_j'\rangle^{n(j)}.
\end{align*}
and thus
\begin{align*}
\bar{H}_I{}&((u_i')_{i\in I\cup I_2(F_I)})=\frac{1}{d!}\sum_{L:I\subseteq I_2(F_L)} w_L' 
\sum_{(k_1,\dots , k_d)\in J_L }  \\
&\quad a_{k_1,\dots,k_d}\prod_{i \in I_1(F_L)} (\langle b_{l_i}^i,u_i'\rangle^{n(i)}-\langle w_{l_i}^i,u_i'\rangle^{n(i)}) \\
&\quad \times \prod_{j \in I_2(F_L) \backslash I} \langle c_{j},u_j'\rangle^{n(j)}.
\end{align*}

On the other hand, $\bar{G}_I$ is given by
\begin{align*}
\bar{G}_I ((u_i')_{i\in I\cup I_2(F_I)})&=\sum_{J: F_J\cap F_I \in \mathcal{F}_0(P)} \gamma(u_j',j\in J)\\
&=\sum_{v \in \mathcal{F}_0(F_I)} \gamma(v,F_I)\\
&= V_0(F_I)\\
& =1.
\end{align*}
where $\gamma(u_j',j\in J)$ is given by the formula \eqref{vinkel} with $u_{i_1^F},\dots, u_{i_1^F}$ replaced by $\{u_j',j\in J\}$.
%
%On the other hand, $\bar{G}_I$ is given by
%\begin{align*}
%\bar{G}_I ((u_i')_{i\in I\cup I_2(F_I)})&=\sum_{F\in \mathcal{F}_q(P'):F\cap F_I\in \mathcal{F}_0(P')} \gamma(F,P')\\
%&=\sum_{v \in \mathcal{F}_0(F_I)} \gamma(v,F_I)\\
%&= V_0(F_I)\\
%& =1.
%\end{align*}

The first equality follows because the index set $J_{v,\sigma}$ in \eqref{GI}
consists of the $j_{d-q+1},\dots, j_d$ such that
\begin{equation*}
\prod_{s=d-q+1}^d  t_{{i_1^F},\dots,{i_{d-q}^F}, {i_{\sigma(d-q+1)}^v},\dots ,{i_{\sigma(s)}^v}}(j_s) = \prod_{i\in I} t_i.
\end{equation*}
In particular,
\begin{equation*}
I\subseteq \{{i_1^F},\dots,{i_{d-q}^F}, {i_{\sigma(d-q+1)}^v},\dots ,{i_{\sigma(s)}^v}\} \subseteq I\cup I_2(F_I).
\end{equation*}
Suppose some $i_k^F\in I$. Then some $j_s=k$. But then 
\begin{equation*}
 a_{j_s}(u_{i_1^F}',\dots,u_{i_{d-q}^F}', u_{i_{\sigma(d-q+1)}^v}',\dots ,u_{i_{\sigma(s)}^v}')=0
 \end{equation*}
by (iv) in Lemma \eqref{Haface}. Therefore, only terms with $\{{i_{d-q+1}^v},\dots ,{i_{s}^v}\}=I$ survive in \eqref{GI}.
This means that $v=F\cap F_I \in \mathcal{F}_0(F_I)$ and the contribution from this $v$ in \eqref{GI} is 1 for each $\sigma\in \Sigma_q$ by (iii) in Lemma \eqref{Haface}.

The second equality follows because if $F_J\cap F_I=v$ then $u_j'$, $j\in J$, are the normal vectors of $F_I$ corresponding to the faces meeting at the vertex $v$.

%\begin{equation*}
%\lin(u_{i_1^F}',\dots,u_{i_{d-q}^F}')=\lin(F_I).
%\end{equation*}

% from \eqref{Haface} because if $v\in F$
%\begin{equation*}
% t_{{i_1^F},\dots,{i_{d-q}^F}, {i_{\sigma(d-q+1)}^v},\dots ,{i_{\sigma(s)}^v}}(j_s)=t_i
%\end{equation*}
% for some $i\in I$, then $i$ must be the $j_s$th coordinate in $ ({i_1^F},\dots,{i_{d-q}^F}, {i_{\sigma(d-q+1)}^v},\dots ,{i_{\sigma(s)}^v})$. But then 
%\begin{equation*}
% a_{j_s}(u_{i_1^F}',\dots,u_{i_{d-q}^F}', u_{i_{\sigma(d-q+1)}^v}',\dots ,u_{i_{\sigma(s)}^v}')=\begin{cases}
% 1, &$j_s=s$\\
% 0, &$j_s\neq s$\\
% \end{cases}
%\end{equation*}
%by Lemma . Hence $\{{i_{d-q+1}^v},\dots ,{i_{s}^v}\}=I$. 

Suppose $d-q$ is odd and $q>0$. Choose a rotation $R\in SO(d)$ changing all signs in $\lin(u_i',i\in I)^\perp$. This is possible because $\dim(\lin(u_i',i\in I)^\perp)=d-q<d$. This clearly preserves $\bar{G}_I=1$ since the orthogonality properties among the $u_i'$ are not changed. Since the $a_{k_1,\dots, k_d}$ are rotation invariant and $d-q$ is odd, $\bar{H}_I$ changes sign.  As $SO(d)$ is connected, there is a path from $(u_i')_{i\in I\cup I_2(F_I)}$ to $(Ru_i')_{i\in I\cup I_2(F_I)}$ inside $W'$. It follows that $\bar{H}_I$ and $\bar{G}_I$ cannot agree everywhere on $W'$. As they are both analytic and $W'$ is connected, $\Ha^{(d-1)N}(W'\cap \{\bar{H}_I = \bar{G}_I\})=0$. This proves the claim in the case where $d-q$ is odd.

If $d-q$ is even, we assume that $\mathcal{US}^N_\mu$ is chosen such that the elements have a $(d-q)$-face which is  combinatorially isotopic to $[0,1]^{d-q}$. Assume that $F_I$ is this face.  Define $\bar{H_{I'}}$ and $\bar{G}_{I'}$ for $I\subsetneq I'$ in a way similar to $\bar{H_{I}}$ and $\bar{G}_{I}$. In particular, $\bar{G}_{I'}=0$. It is enough to show that 
\begin{equation*}
\Ha^{(d-1)N}(W' \cap \{\bar{H}_{I'}= \bar{G}_{I'}\} )=0
\end{equation*}
for some $I\subseteq I'$ since $V^N_\mu \subseteq A:=\bigcap_{I\subseteq I'} \{ \bar{H}_{I'}= \bar{G}_{I'}\}$.

Since $(u_i')_{i\in I_2(F_I)}$ is exactly the set of normal vectors of $F_I $ considered as a subset of $\aff(F_I)$, it is possible to choose a path from $(u_i')_{i\in I_2(F_I)}$ inside $\lin(F_I)$ to $(u_i'')_{i\in I_2(F_I)}$ such that these are the normal vectors $\{\pm v_1,\dots,\pm v_{d-q}\}$ of the or\-tho\-go\-nal box $\bigoplus_{i=1}^{d-q} [0,v_i]$. This ensures that for all $v\in \mathcal{F}_0(F_I)$, the $u_i''$ having $i\in I_1(v)$ are orthogonal. 

For $(k_1,\dots,k_d)\in J_L$, $\bigcap_s F_{k_s}\subseteq F_I$. By (iii) in Lemma~\ref{polvol}, the only terms in \eqref{ajformel} contributing to $a_{k_1,\dots,k_d}$ have $j_s=s$ and $i_{\sigma(s)}^v=k_s$ for all $s$. But then $\{k_1,\dots,k_d\}=I_1(v)$ for some $v\in \mathcal{F}_0(F_I)$ and $\sigma$ is uniquely determined. In this case $a_{k_1,\dots,k_d}=1$, again by Lemma \ref{polvol} (iii). Hence 
\begin{align*}
\bar{H}_I{}&((u_i'')_{i\in I\cup I_2(F_I)})=N_I \sum_{L:I\subseteq  I_2(F_L)} w_L' \\
&\times \prod_{i \in I_1(F_L)} (\langle b_{l_i}^i,u_i''\rangle-\langle w_{l_i}^i,u_i''\rangle) \prod_{j \in I_2(F_L) \backslash I} \langle c_{j},u_j''\rangle
\end{align*}
where $N_{I}=|\mathcal{F}_0(F_I)|$.

A similar argument for $I'$ with $I\subseteq I' \subseteq I\cup I_2(F_I)$ shows that the coefficient $\bar{H}_{I'}$ in front of $\prod_{i\in I'} t_i$ is 
\begin{align}\label{barH}
\bar{H}_{I'}{}&((u_i'')_{i\in I\cup I_2(F_I)})=N_{I'}\sum_{L:I'\subseteq  I_2(F_L)} w_L' \\
&\times \prod_{i \in I_1(F_L)} (\langle b_{l_i}^i,u_i''\rangle-\langle w_{l_j}^i,u_i''\rangle) \prod_{j \in I_2(F_L) \backslash I'} \langle c_{j},u_j''\rangle.\nonumber
\end{align}

Suppose $(u_i'')_{i\in I\cup I_2(F_I)}\in A$. Then \eqref{barH} vanishes. If $I\subsetneq I'$, multiplication by $\prod_{j \in I'\backslash I} \langle c_{j},u_j''\rangle$ shows that also
\begin{align*}
\bar{K}_{I'}{}&((u_i'')_{i\in I\cup I_2(F_I)}):=N_I\sum_{L:I'\subseteq I_2(F_L)} w_L' \\
&\times \prod_{i \in I_1(F_L)} (\langle b_{l_i}^i,u_i''\rangle-\langle w_{l_i}^i,u_i''\rangle) \prod_{j \in I_2(F_L) \backslash I} \langle c_{j},u_j''\rangle\\
 ={}& 0.
\end{align*}
 Hence, on $A$
\begin{align*}
\bar{H}_I{}&((u_i'')_{i\in I\cup I_2(F_I)}) = \sum_{I\subsetneq I'} (-1)^{|I'|-|I|+1} \bar{K}_{I'} \\
&+ N_I\sum_{L:I =  I_2(F_L)} w_L' \prod_{i \in I_1(F_L)} (\langle b_{l_i}^i,u_i''\rangle-\langle w_{l_i}^i,u_i''\rangle) \\
=& N_I\sum_{L:F_I \cap  F_L \in \mathcal{F}_0(F_I)} w_L' \prod_{i \in I_1(F_L)} \langle b_{l_i}^i-  w_{l_i}^i,u_i''\rangle  \\
=& N_I\sum_{l_{1},\dots, l_{ d-q}=1}^{2^{n^d}-1} w_{l_1,\dots, l_{d-q}}' \prod_{j=1}^{d-q}
\sum_{\eps_{j} \in \{\pm 1\}}\\ & \langle b_{l_j}(\eps_jv_j) - w_{l_j}(\eps_jv_j),\eps_jv_j\rangle
\end{align*}
where $b_{l_j}(\eps_jv_j)=b_{l_j}^i$ and $w_{l_j}(\eps_jv_j)=w_{l_j}^i$ if $\eps_jv_j=u_i''$.

For $l\in \{1,\dots 2^{n^d}-1\}$, let 
\begin{align*}
\alpha(l)={}&N_I\sum_{l_{2},\dots, l_{ d-q}=1}^{2^{n^d}-1} w_{l, l_2, \dots, l_{d-q}}' \sum_{\eps_2, \dots ,\eps_{d-q} \in \{\pm 1\}}\\
& \prod_{j=2}^{d-q} \langle b_{l_j}(\eps_jv_j)- w_{l_j}(\eps_jv_j),\eps_jv_j\rangle.
\end{align*}
This depends only on $l$ and $ v_j$ for $j=2,\dots ,d-q$.
Then  
\begin{align*}
\bar{H}_I{}&(( u_i'')_{i \in I \cup I_2(F_I)})\\
&=\sum_{l=1}^{2^{n^d}-1} \alpha(l) (\langle b_l^{j_1}-w_l^{j_1},v_{1}\rangle +
 \langle b_{l}^{j_2}-w_{l}^{j_2},-v_{1}\rangle )\\
 &=\langle x, v_{1}\rangle.
\end{align*}
where $v_{1}=u_{j_1}''$ and $-v_{1}=u_{j_2}''$ and $x\in \R^d$ is some vector depending only on $v_2,\dots , v_{d-q}$. It follows  that 
\begin{equation}\label{Hid}
\bar{H}_I((Ru_i'')_{i \in I \cup I_2(F_I)}) = \langle x, Rv_{1}\rangle 
\end{equation}
for any rotation $R\in SO(K^\perp)$, where $SO(K^\perp)$ is the subgroup of $SO(d)$ that fixes $K=\lin(v_2,\dots , v_{d-q})$. But $v_{1}$ is orthogonal to $K$ and $\dim K^\perp =q+1>1$, so \eqref{Hid} cannot equal $\bar{G}_I((Ru_i'')_{i \in I \cup I_2(F_I)})=1$ for all rotations $R\in SO(K^\perp)$. Thus, there must be an $R\in SO(K^\perp)$ such that $(Ru_i'')_{i \in I \cup I_2(F_I)}\notin A$. But then 
\begin{equation*}
H_{I'}((Ru_i'')_{i \in I \cup I_2(F_I)})\neq G_{I'}((Ru_i'')_{i \in I \cup I_2(F_I)})
\end{equation*}
 for at least one $I\subseteq I'$. Since $SO(K^\perp)$ is path connected, $(Ru_i'')_{i \in I \cup I_2(F_I)}\in W'$ and it follows that 
\begin{equation*} 
\Ha^{(d-1)N}(W'\cap \{H_{I'}=G_{I'}\})=0
\end{equation*}
as in the odd case.\qed
\end{proof}

The theorem does not explicitly construct the polytopes for which $\hat{V}_q$ is biased. However, consider the space of orthogonal boxes 
\begin{equation*}
B(U,t)=\bigoplus_{i=1}^d [0,t_iu_i]
\end{equation*}
parametrized by $U\in SO(d)$ and $t\in (0,\infty)^d$.
%
%A parallelotope is, up to translation, a set of the form $\bigoplus_{i=1}^d [o,v_i] \in \mathcal{P}^d$ for  $v_1,\dots,v_d \in \R^d$ linearly independent. We identify the set of parallelotopes with $Gl(d)/ \Sigma_d$ and use the Hausdorff measure on $\R^{d^2}$ to put a measure on the space of polytopes. Then we have:

\begin{corollary}\label{box}
Let $\hat{V}_q$ be a local algorithm for $V_q$ where $1\leq q \leq d-1$. Then $\hat{V}_q(B(U,t))$ is asymptotically biased for almost all $(U,t) \in SO(d) \times (0,\infty)^d$.
\end{corollary}

\begin{proof}
This follows from the proof of Theorem \ref{main1} in the case $d-q$ even since the proof does not use the fact that $d-q$ is even, only that $q\neq 0,d$. \qed
%
%For a general $P(V)$, since $\Ha^q(F)=\prod_{s=1}^q |v_{i_s}|$ for all $F\in \mathcal{F}_q(P)$ of the form $ \bigcap_{i\neq i_1,\dots, i_q} F_{\eps_i i}$ for $\eps_i \in \{\pm 1\}$ and  $\sum_{\eps_i\in \{ \pm \}} \gamma(\pos( \eps_i v_i, i \neq i_1,\dots ,i_q ) )$  
%\begin{equation*}
%V_q(P(V))= \sum_{1\leq i_1<\dots <i_q\leq d} \prod_{s=1}^q |v_{i_s}| .
%\end{equation*} 
%On the other hand, let $V\in Gl(d)$ be a linear map taking $v_i$ to the $i$th standard basis vector $e_i$. Then $N_l(P(V)\cap a\La_c)= N_l(VP \cap aV(\La_c))$ and hence $\hat{V}^{a\La_c}_q(P(V))=\hat{V}^{aV(\La_c)}_q(VP(V))=\hat{V}^{aV(\La_c)}_q(P(I))$ where $\hat{V}^{aV\La_c}_q$ is the estimator with the same weights as $\hat{V}^{a\La_c}_q$, but with respect to the lattice $V(\La )$. Hence $E\hat{V}^{a\La_c}_q(P(V))=E\hat{V}^{aV(\La_c)}_q(P(I))$.
%
%But for almost all $R\in SO(d)$, $E\hat{V}^{aU(\La + c)}_q(RP(\Z^d))= E\hat{V}^{aU(\La + c)}_q(RP(\Z^d))\neq V_q(RP(\Z^d))=V_q(P(U,b)$. This implies that almost all $\hat{V}^{a(\La + c)}_q(U^{-1}RP(U^{-1}R))$ are biased. 
\end{proof}

\begin{remark}
It seems likely that Theorem \ref{main1'} should hold for all combinatorial isotopy classes of simple polytopes in the case $d-q$ even as well, but a proof would require a different argument.
\end{remark}

\subsection{The Euler characteristic in 2D}\label{euler}
In this section we investigate the estimation of the Euler characteristic $V_0$ on $\mathcal{P}^2$ and prove Theorem \ref{main2} in the case $d=2$. 

From Section \ref{polHM} we have:
\begin{corollary}\label{2Dvol}
Let $P\in \mathcal{P}^{2,N}$ be given and let $\theta_{ij}$ denote the interior angle between $F_i$ and $F_j$, i.e.\ $\pi - \theta_{ij}$ is the angle between $u_i$ and $u_j$. For $a$ sufficiently small, 
%Let $P=\bigcap_{i=1}^N H^-_{u_i,x_i} \subseteq \R^2$ be a polygon.
%On $\tilde{\mathcal{E}^{d,N}}$
\begin{align}\nonumber
E{}&\hat{V}_0(P)= \sum_{l=1}^{2^{n^d}-1}  w_l'(a)\sum_{i=1}^N \bigg((-h(B_l\oplus\check{W}_l,u_i))^+\\\nonumber
&\quad \times \bigg(a^{-1}\Ha^1(F_i)+\sum_{j\in I_2(F_i)} h(\check{C}_{0,0}^n,u_{j})\csc (\theta_{ij})\bigg) \\ \nonumber
&+\frac{1}{2}\delta_l(u_i)(h(\check{W}_l,u_i)^2-h(B_l,u_i)^2)\sum_{j\in I_2(F_i)}\cot (\theta_{ij})\bigg)\\\label{2Deu}
&+\sum_{l,k=1}^{2^{n^d}-1} w_{lk}'(a) \sum_{v\in \mathcal{F}_0(P)} \csc (\theta_{i_1^vi_2^v})\\\nonumber
&\quad \times (-h(B_l\oplus \check{W}_l,u_{i_1^v}))^+(-h(B_k\oplus \check{W}_k,u_{i_2^v}))^+
\end{align}
where $\delta_{l}(u)=\mathds{1}_{\{h(B_l\oplus \check{W}_l,u)<0\}}$ and $w_l'$ is as in the proof of Theorem \ref{main1'}.
\end{corollary}

As usual $\csc$ denotes the function $\frac{1}{\sin}$.

\begin{proof}
For $a$ sufficiently small, no more than two of the sets $X_{i,l}$ with $l\neq 0,2^{n^d}-1$ can intersect and if $\{i,j\}$ does not equal $I_1(v)$ for any vertex $v$ of $P$, then 
\begin{equation*}
X_{i,l}\cap X_{j,k}\subseteq X_{m,0}
\end{equation*}
 for some $m$. Thus, to use Corollary \ref{brik}, we need only compute:
\begin{align*}
\Ha^2{}&(X_{i,l}\cap \bigcap_{j\neq i} X_{j,2^{n^d}-1})=(-h(B_l\oplus\check{W}_l,u_i))^+\\
&\quad \times \bigg(a\Ha^1(F_i)+a^2\sum_{j\in I_2(F_i)} h(\check{C}_{0,0}^n,u_{j})\csc (\theta_{ij})\bigg) \\
& +\frac{a^2}{2}\delta_l(u_i)(h(\check{W}_l,u_i)^2-h(B_l,u_i)^2)\sum_{j\in I_2(F_i)}\cot (\theta_{ij}),\\
\Ha^2{}&(X_{i_1^v,l}\cap X_{i_2^v,k} \cap \bigcap_{m\neq i_1^v,i_2^v} X_{m,2^{n^d}-1})=a^2\csc (\theta_{i_1^v i_2^v})\\
&\quad \times (-h(B_l\oplus \check{W}_l,u_{i_1^v}))^+(-h(B_k\oplus \check{W}_k,u_{i_2^v}))^+.
\end{align*}
This follows from Lemma \ref{volpiece} or directly from plane geometric considerations, see Figure \ref{fi}. 
%The set $X_{i,l}\cap \bigcap_{j\neq i} X_{j,2^{n^d}-1}$ is a trapezium while $X_{i_1^v,l}\cap X_{i_2^v,k} \cap \bigcap_{m\neq i_1^v,i_2^v} X_{m,2^{n^d}-1}$ is a parallelogram. \qed
\end{proof}

\begin{figure}
\begin{equation*}
\setlength{\unitlength}{1cm}
%\begin{equation*}
\begin{picture}(7,4)
\put(0,3){\line(1,0){7}}
\put(0,3.8){\line(1,0){7}}
\put(0.5,3.3){$X_{i_1^v,l}$}
\put(0.7,1){$X_{i_2^v,k}$}
\put(3.7,1){$P$}
\put(4,3.3){$(-h(B_l\oplus \check{W}_l,u_{i_1^v}))^+$}
\put(2.9,1.8){$\theta_{i_1^vi_2^v}$}
\put(3.8,3){\line(0,1){0.8}}
\put(0,0){\line(1,2){2}}
\put(1,0){\line(1,2){2}}
\put(1.6,0.4){\line(1,2){1}}
\put(1.6,0.4){\line(1,0){5}}
\put(2.6,2.4){\line(1,0){3}}
\put(6.6,0.4){\line(-1,2){1}}
\qbezier(2.4,2)(3,2)(3,2.4)
\end{picture}
\end{equation*}
\caption{Illustration of the sets $X_{i_1^v,l}$ and $X_{i_2^v,k}$.}\label{fi}
%\end{equation*}
\end{figure}
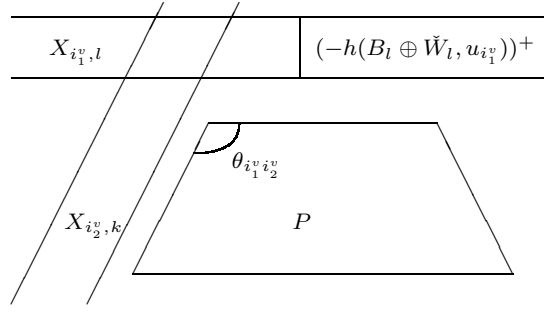

We introduce the following notation:
\begin{definition}
Let $\La\subseteq \R^2$ be the lattice spanned by $\xi=\{\xi_1,\xi_2\}$.
Define $D_\La \subseteq S^1$ by
\begin{equation*}
D_\La=\left\{\tfrac{z}{|z|}\mid z\in C_{-n(\xi_1+\xi_2),0}^{2n} \backslash\{0\}\right\}. 
\end{equation*}
We say that a vertex $v$ of a polygon $P$ is $n$-critical if $(P-v)\cap aC_{-n(\xi_1+\xi_2),0}^{2n}=\{0\}$ for all $a$ small enough or equivalently if $a^{-1}(P-v)\cap S^1$ is contained in a connected component of $S^1\backslash D_\La$.
\end{definition}

We can now prove Theorem \ref{main2} in the case $d=2$.
%\end{theorem}

\begin{proof}(Theorem \ref{main2} for $d=2$)
Suppose the weights $w_l(a)$ of an asymptotically unbiased estimator $\hat{V}_0$ are given. We just need to show the existence of one element in $\mathcal{P}^{2,N}\backslash \mathcal{V}^{N}$ for some $N$, so assume for contradiction that $\mathcal{V}^{N}=\mathcal{P}^{2,N}$. Since all polygons are simple, Corollary \ref{weights} allows us to assume that the weights are homogeneous, i.e.\ $w_l(a)=w_l$.

Let 
\begin{align*}
v_1&=(\cos\varphi,\sin\varphi),\\
v_2&=(\cos(\varphi+\psi),\sin(\varphi+\psi)),
\end{align*}
where $(\varphi,\psi)\in U$ for some small open subset $U \subseteq \R^2$ such that $v_1$ and $v_2$ lie in the same connected component $E\subseteq S^1\backslash D_\La$.

Consider a parallelogram 
\begin{equation}\label{parallel}
P(\varphi,\psi,s_1,s_2)=[0,s_1v_1]\oplus[0,s_2v_2]
\end{equation}
for $s_1,s_2>0$.
 Then $P$ has two $n$-critical vertices at 0 and $s_1v_1+s_2v_2$. The normal vectors of $P$ are 
\begin{align*}
u_1=-u_3&= (-\sin \varphi, \cos \varphi),\\
u_2=-u_4&=(-\sin (\varphi+\psi), \cos (\varphi+\psi)).
\end{align*}

Observe that $\csc (\theta_{i_1^v i_2^v}) = \csc \psi$ for all $v\in \mathcal{F}_0(P)$, and if $I_2(F_i)=\{j_1,j_2\}$, then $\cot (\theta_{ij_1})=-\cot (\theta_{ij_2})$. 

Since $\lim_{a\to 0} E\hat{V}_0(P)$ exists, the coefficient in front of $a^{-1}$ in \eqref{2Deu}
\begin{equation*}
\sum_{l=1}^{2^{n^d}-1}  w_l'\sum_{i=1}^2 s_i\sum_{\eps =\pm 1}(-h(B_l\oplus\check{W}_l,\eps u_i))^+
\end{equation*}
must vanish.
This holds for all $s_1,s_2>0$, so for each $i=1,2$, also
\begin{equation*}
\sum_{l=1}^{2^{n^d}-1}  w_l' ((-h(B_l\oplus\check{W}_l,u_i))^++(-h(B_l\oplus\check{W}_l,-u_i))^+)=0
\end{equation*}
 and Corollary \ref{2Dvol} reduces to
\begin{align*}
E\hat{V}_0(P)={}&\csc \psi \sum_{l,k=1}^{2^{n^d}-1} w_{lk}'  \sum_{v\in \mathcal{F}_0(P)}\\
&(-h(B_l\oplus \check{W}_l,u_{i_1^v}))^+(-h(B_k\oplus \check{W}_k,u_{i_2^v}))^+
\end{align*}
for all $a$ sufficiently small.

Let $R$ denote the reflection of $C_{0,0}^n$ in the point $(\tfrac{n-1}{2}\xi_1,\tfrac{n-1}{2}\xi_2)$ and observe that 
\begin{equation*}
h(B_l\oplus \check{W_l},u)=h(RB_l\oplus (-RW_l),-u).
\end{equation*}
 Thus, since the weights are reflection invariant,
\begin{align}\nonumber
 E{}&\hat{V}_0(P)
= 2\csc \psi\sum_{l,k=1}^{2^{n^d}-1}(w'(B_l \cap B_k)+w'(RB_l \cap B_k))\\
&\times(- h(B_l\oplus \check{W}_l,u_1))^+(-h(B_k\oplus \check{W}_k,u_2))^+.\label{sums}
\end{align} 
for all sufficiently small $a$. 
 
Let $\beta_l^+,\omega_l^-:S^{1}\to C_{0,0}^n$ be functions such that $\langle \beta_l^+(u),u \rangle =h(B_l, u)$ and $\langle\omega_l^-(u), u\rangle =-h(\check{W}_l, u)$. In par\-ti\-cu\-lar, $ h(B_l\oplus \check{W}_l,u)=\langle \beta_l^+(u)-\omega_l^-(u),u\rangle $. Note that $\beta_l^+$ and $\omega_l^-$ are constant on the set  $R_{-\frac{\pi}{2}}E\subseteq S^1$ where $R_{-\frac{\pi}{2}}$ is the rotation by ${-\frac{\pi}{2}}$. Thus,  whenever $\varphi,\varphi+\psi\in E$,
\begin{align*} 
&\delta_{l}(u_1)=\delta_{l}(u_2),\\
&\beta_l=\beta_l^+(u_1)=\beta_l^+(u_2),\\
&\omega_l=\omega_l^-(u_1)=\omega_l^-(u_2),
\end{align*}
for some fixed vectors $\beta_l,\omega_l\in \R^2$.

Write $\omega_l-\beta_l=(x_l,y_l)$. Then  for $\varphi,\varphi+\psi\in E$,
\begin{align*}
(-h(&B_l\oplus \check{W_l},u_1))^+(-h(B_k\oplus \check{W}_k,u_2))^+\\
&+(-h(B_{k}\oplus \check{W}_k,u_1))^+(-h(B_l\oplus \check{W_l},u_2))^+\\
=& \delta_{l}(u_1)\delta_{k}(u_1)(\langle \omega_l-\beta_l,u_1\rangle\langle \omega_k-\beta_k,u_2\rangle \\
& + \langle \omega_k-\beta_k,u_1\rangle \langle \omega_l-\beta_l,u_2\rangle)\\
=&\delta_{l}(u_1)\delta_{k}(u_1)((-x_l\sin \varphi +y_l\cos \varphi) \\
&\quad\times (-x_{k}\sin (\varphi+\psi) + y_{k}\cos (\varphi+\psi))\\
&+(-x_{k}\sin \varphi +y_{k}\cos \varphi)\\
&\quad\times (-x_{l}\sin (\varphi+\psi) + y_{l}\cos (\varphi+\psi)))\\
=&\delta_{l}(u_1)\delta_{k}(u_1)((2x_lx_{k}\sin \varphi \sin(\varphi+\psi)\\
& +2y_ly_{k}\cos \varphi \cos(\varphi+\psi))-(x_{k}y_l +x_ly_{k})\\
&\quad\times (\sin \varphi \cos(\varphi+\psi)+\cos \varphi \sin (\varphi+\psi))).
\end{align*}

Since $w(B_l\cap B_k) = w(B_k\cap B_l)$ and 
\begin{equation*}
w(RB_l\cap B_k) =w(R( RB_l\cap B_k))=w( RB_k\cap B_l),
\end{equation*}
the terms in \eqref{sums} pair up, showing that $E\hat{V}_0(P(\varphi,\psi))$ is a linear combination of the functions
\begin{align}
\begin{split}\label{functions}
&\cos  \varphi \cos(\varphi+\psi)\csc \psi \\
&\qquad \qquad=\cos^2 \varphi \cot\psi - \sin\varphi \cos\varphi,\\ 
&\sin \varphi \sin(\varphi+\psi)\csc \psi \\
 &\qquad \qquad=\sin^2 \varphi \cot\psi + \sin\varphi \cos \varphi,\\ 
&(\cos \varphi \sin(\varphi+\psi)+\sin \varphi \cos (\varphi+\psi)) \csc \psi \\
&\qquad \qquad = \sin \varphi \cos \varphi \cot \psi+ \cos^2 \varphi -\sin^2 \varphi.
\end{split}
\end{align}

On the other hand, \eqref{sums} equals $V_0(P(\varphi,\psi))=1$ for all $(\varphi,\psi)\in U$. But the functions in \eqref{functions} are clearly linearly independent of the constant function $1$, yielding the contradiction.\qed
%, as one may see as follows:
%Interchanging $(\varphi,\psi)$ with $(\pi-\varphi,-\psi)$ and comparing sign changes shows that 1 must be a multiplum of
%\begin{equation*}
%(\cos \varphi \sin(\varphi+\psi) + \sin \varphi \cos(\varphi+\psi))\csc \psi 
%= (\cos^2 \varphi -\sin^2 \varphi)+  \sin \varphi \cos \varphi \cot \psi,
%\end{equation*}
%which is clearly not the case. 
%Thus a linear combination of \eqref{functions} cannot equal 1 for all $(\varphi,\psi)\in U$. 
\end{proof}

\begin{corollary}\label{worst}
Any local estimator for ${V}_0$ has a worst case asymptotic relative bias on $\mathcal{P}^2$ of at least 1.
\end{corollary}

\begin{proof}
Let $P(\varphi,\psi)$ be as in the proof of Theorem \ref{main2} for $d=2$. The proof shows that $\lim_{a\to 0}E\hat{V}_{0}(P(\varphi,\psi))$ has the form
\begin{align}
\begin{split}\label{aer}
\alpha_1{}&(\cos^2 \varphi \cot\psi - \sin\varphi \cos\varphi ) \\
&+\alpha_2(\sin^2 \varphi \cot\psi + \sin\varphi \cos \varphi)\\
 &+ \alpha_3(  \sin \varphi \cos \varphi \cot \psi +
\cos^2 \varphi -\sin^2 \varphi)
\end{split}
\end{align}
for some $\alpha_1,\alpha_2,\alpha_3\in \R$ and all $(\varphi,\psi)\in I\times (0,\eps)\subseteq U$ for   some small open interval $I$ and some $\eps>0$.

The functions $\cos^2\varphi$, $\sin^2\varphi$, and $\sin\varphi \cos \varphi$ are li\-ne\-ar\-ly independent, so if \eqref{aer} is non-trivial, there must be a $\varphi\in I$ such that 
\begin{equation*}
\lim_{\psi \to 0}\lim_{a\to 0}E\hat{V}_{0}(P(\varphi,\psi)) = \pm \infty.
\end{equation*}
\qed
\end{proof}

Note how the fact that $P(\varphi,\psi)$ had an $n$-critical vertex was essential in the proof.
The next proposition shows that the polygons with $n$-critical vertices are the only sets in $\mathcal{P}^2$ where the estimation of $V_0$ fails. To get a slightly more general result, we first extend the definition of an $n$-critical vertex to the class $\mathcal{K}^2$ of compact convex sets with non-empty interior.

\begin{definition}
Let $K\in \mathcal{K}^2$. We say that $x\in \partial K$ is an $n$-critical boundary point if for all $a>0$, 
\begin{equation*}
(K-x)\cap aC_{-n(\xi_1+\xi_2),0}^{2n}=\{0\}.
\end{equation*}
\end{definition}
Note that $K$ can have at most finitely many $n$-critical boundary points.

\begin{lemma}\label{convexity}
Let $K\in \mathcal{K}^2$ have no $n$-critical boundary points. Then there exists a $\delta >0$ such that  whenever $a<\delta$,
\begin{equation}\label{convexnc}
(K-x)\cap aC_{-n(\xi_1+\xi_2),0}^{2n}\neq\{0\}
\end{equation}
for all $x\in \partial K$.
\end{lemma}

\begin{proof}
Let $x\in \partial K$. Then there is an $a(x)> 0$ depending on $x$ such that $[x,x+a(x)c]\subseteq K$ for some $c\in C_{-n(\xi_1+\xi_2),0}^{2n}$. There is an open neighborhood $U_x$ of $x$ in $\partial K$ such that $y+\frac{1}{2}a(x)c\in K$ for all $y\in U_x$.  Cover $\partial K$ by finitely many such $U_x$ and choose $a$ to be the smallest of the corresponding $\frac{1}{2}a(x)$. \qed
\end{proof}

Let $(B_l^n,W_l^n)$ be a configuration. Define the corresponding weight 
\begin{equation} \label{weight}
w_l=\sum_{k=1}^{n^2} (-1)^k\frac{1}{k}n_l^{k-1}
\end{equation}
where $n_l^k$ is the number 
\begin{equation*}
n_l^k=\bigg\vert\bigg\{S\subseteq \La\backslash \{0\} \bigg\vert \vert S \vert =k, B_l \cap \bigcap_{z\in S} C_{z,0}^n  \neq \emptyset \bigg\}\bigg\vert.
\end{equation*}

\begin{proposition}\label{nalg}
Let $\hat{V}_0^n$ be the local algorithm based on $n \times n$ configurations with weights given by \eqref{weight}. For all $K\in \mathcal{K}^2$ with no $n$-critical boundary points, $\hat{V}_0^n(K)=1$ whenever $a$ is sufficiently small.
\end{proposition}

The idea is to approximate $K$ by a polyconvex set.
Let $P_z=\conv(C_{z,0}^n\cap K)$ be the convex hull of $C_{z,0}^n\cap K$ and define the approximation
\begin{equation*}
\hat{K}=\bigcup_{z\in \Z^2} P_z.
\end{equation*}
Then the proof will show that $V_0(K)=V_0(\hat{K})$ and that $\hat{V}_0^n(K)=V_0(\hat{K})$.

\begin{proof}
Let $K\in \mathcal{K}^2$ with no $n$-critical boundary points be given. For simplicity, assume $\La = \Z^2$. The general case follows by considering a linear map $L:\R^2 \to \R^2$ with $L(\La)=\Z^2$. Then $K$ has an $n$-critical vertex for $\La$ if and only if $L(K)$ has an $n$-critical vertex with respect to $\Z^2$.  Moreover, $N_l(K\cap \La)=N_l(L(K)\cap \Z^2)$ and thus $E\hat{V}_0^{a\La}(K)=E\hat{V}_0^{a\Z^2}(L(K))$.
Choose $a$ so small that \eqref{convexnc} is satisfied and such that $K$ contains a ball of radius $\sqrt{2}(n+1)a$. By possibly considering $a^{-1}K$ instead of $K$, we may assume that $a=1$ to keep notation simple. 

We first claim that 
\begin{equation*}
V_0(\hat{K})=V_0(K)=1.
\end{equation*}
For this, it is enough to show that $\hat{K}$ and $\R^2\backslash \hat{K}$ are both connected.

In order to show that $\hat{K}$ is connected, we show that every $x=(x_1,x_2) \in \hat{K}\cap \Z^2$ is connected by a path in $\hat{K}$ to a fixed reference point $y=(y_1,y_2)\in \hat{K}\cap \Z^2$ with $y + B(\sqrt{2}n)\subseteq K$. We may assume that $x_1\leq y_1$ and $x_2\leq y_2$ such that $C_{0}^{n}\cap (K-x) \neq \{0\}$. Then $C_{0,0}^{n}\cap (K-x)$ must contain a point $z\neq x$. 
To see this, choose $p\in \partial K $ with $x\in [p,y]$. Then $C_{0,0}^n\cap (K-p)$ contains a point $z\neq 0$ since $p$ is not $n$-critical. Since $p+z,y+z\in K$, also $x+z\in K$ by convexity. Thus $x$ is connected to $x+z$ in $\hat{K}$, and the claim follows by induction on $|x_1-y_1|+|x_2-y_2|$.

To show that $\R^2\backslash \hat{K}$ is connected, assume for contradiction that $x\in K\backslash \hat{K}$ is contained in a compact component. Let $l$ be the vertical line through $x$. Let $b_1,b_2\in \hat{K}\cap l$ be such that 
\begin{equation}\label{bs}
[b_1,b_2]\cap \hat{K} =\{b_1,b_2\},
\end{equation}
$x\in [b_1,b_2]$, and the vector $b_2-b_1$ points upwards.   Then for $i=1,2$ there are line segments $[x_i,y_i]\subseteq \partial P_{z_i}$ with $x_i,y_i,z_i\in \Z^2$ such that $b_i\in [x_i,y_i]$.
 After possibly reflecting the picture in the coordinate axes, we may assume:
\begin{align*}
&\aff[x_1,y_1] \cap \aff[x_2,y_2] \subseteq H^-_{-e_1,-\langle x,e_1\rangle},\\
&x_1,x_2 \in H^-_{e_1,\langle x,e_1\rangle},\\
&y_1,y_2 \in H^-_{-e_1,-\langle x,e_1\rangle},\\
&\langle x_1,e_1 \rangle \leq \langle x_2,e_1 \rangle.
\end{align*}

First observe that the vertical distance from $x_2$ to $[x_1,y_1]$ is at most $1$. Assume this were not true. If $[x_1,y_1]$ has positive slope, either $x_2=y_1+(0,m)\in l$ for some $m\in \N$, implying $x\in [x_2,y_1]\subseteq \hat{K}$, or there is an $m\in \N$ such that $x_2-(0,m)$ lies above $[x_1,y_1]$ and $\conv(x_2-(0,m),x_1,y_1) \subseteq P_{z}$ for some $z\in \Z^2$. If $[x_1,y_1]$  has non-positive slope, so must $[x_2,y_2]$ and hence $\conv(x_2,x_2-(0,1),y_2) \subseteq P_{z}$ for some $z\in \Z^2$. All three cases contradict the assumption \eqref{bs}.
 
This implies that either $\conv(x_1,y_1,x_2)\subseteq P_z$ for some $z\in \La$, or $x_2=y_1+(0,1)$, or $[x_1,y_1]$ has negative slope and $x_2=x_{1}+(0,1)$.  The second case implies $[x_2,y_1]\subseteq \hat{K}\cap l$. In the third case, $[x_2,y_2]$ must also have negative slope and hence $\conv(x_2-(0,1),x_2,y_2)\subseteq  P_z$ for some $z\in \La$. Again, all three cases contradict the assumption \eqref{bs}. 
 
The proof is now complete if we can show that 
\begin{equation*}
\hat{V}_0^n(K\cap \Z^2)=V_0(\hat{K}).
\end{equation*}
By the inclusion-exclusion principle,
\begin{align*}
V_0(\hat{K})&=\sum_{k=1}^{n^2} \sum_{\substack{S\subseteq \La,\\ \vert S \vert = k}} (-1)^kV_0\bigg(\bigcap_{z\in S} {P}_z\bigg)\\
&=\sum_{k=1}^{n^2}\sum_{z_0\in \La} \sum_{\substack{S\subseteq \La\backslash \{z_0\},\\ \vert S \vert = k-1}} (-1)^k\frac{1}{k}V_0\bigg(P_{z_0} \cap \bigcap_{z\in S} {P}_{z}\bigg)\\
&=\sum_{k=1}^{n^2}(-1)^k\frac{1}{k}\sum_{z_0\in \La} \sum_{\substack{S\subseteq \La\backslash \{z_0\},\\ \vert S \vert = k-1}} \mathds{1}_{\{P_{z_0} \cap \bigcap_{z\in S} {P}_{z}\neq \emptyset\}}.
\end{align*}
 
On the other hand, the weights are constructed such that
\begin{align*}
\hat{V}_0{}&(K)= \sum_{z_0\in \La }\sum_{l=1}^{2^{n^2}-1} w_l \mathds{1}_{\{C_{z_0,0}^n\cap K= z_0+ B_l\}}\\
&= \sum_{z_0\in \La }\sum_{l=1}^{2^{n^2}-1} \sum_{k=1}^{n^2} (-1)^k\frac{1}{k}n_l^{k-1} \mathds{1}_{\{C_{z_0,0}^n\cap K= z_0+ B_l\}}\\
&= \sum_{z_0\in \La }\sum_{k=1}^{n^2} (-1)^k\frac{1}{k} \sum_{\substack{S \subseteq \La\backslash \{z_0\},\\ |S|=k-1 }} \mathds{1}_{\{\bigcap_{z\in S} C_{z,0}^n \cap (C_{z_0,0}^n\cap K) \neq \emptyset\}}
\end{align*}
so it remains to show that if $P_{z_1}\cap \dots \cap P_{z_k}\neq \emptyset$, then $C_{z_1,0}^n\cap \dots \cap C_{z_k,0}^n \cap K \neq \emptyset$.

For $k=1$ this is trivial. Assume $P_{z_1}\cap P_{z_2}\neq \emptyset$. If $P_{z_1} \subseteq P_{z_2}$, then the claim is clearly true. Otherwise, $\partial P_{z_1}\cap \partial P_{z_2}\neq \emptyset$. Hence there are $x_i,y_i\in C_{z_i,0}^n\cap K$ such that the line segments $[x_1,y_1]$ and $[x_2,y_2]$ intersect in $C_{z_1}^n\cap C_{z_2}^n$. Assume $x_1,y_1,x_2,y_2 \notin  C_{z_1}^n\cap C_{z_2}^n$. Then $[x_1,y_1]$ divides $C_{z_2}^n$ into two components $C^1$ and $C^2$ with $C^1\subseteq C_{z_1}^n\cap C_{z_2}^n$. As $[x_2,y_2]$ intersects $[x_1,y_1]\cap C_{z_2}^n $, either $x_2$ or $y_2$ must belong to $C^1\cup [x_1,y_1]\subseteq C_{z_1}^n$, which is a contradiction.

For $k\geq 3$, assume that $z_1$ and $z_2$ have the smallest and largest 1st coordinate among the $z_i$, respectively. By the above, there is a $y_1\in C_{z_1,0}^n\cap C_{z_2,0}^n\cap K$. If $y_1 $ lies in $ C_{z_1,0}^n\cap \dots \cap C_{z_k,0}^n$, we are done. Otherwise, suppose that the 2nd coordinate is too large for $y_1$ to belong to $C_{z_1,0}^n\cap \dots \cap C_{z_k,0}^n$. Let $z_3$ have the smallest 2nd coordinate among the $z_i$. There are points 
\begin{align*}
&y_2\in C_{z_1,0}^n\cap C_{z_3,0}^n\cap \dots \cap C_{z_k,0}^n\cap K,\\
&y_3\in C_{z_2,0}^n\cap C_{z_3,0}^n\cap \dots \cap C_{z_k,0}^n\cap K,
\end{align*}
 by induction. If $ y_2, y_3 \notin C_{z_1}^n\cap \dots \cap C_{z_k}^n $, write $y_i=(r_i,s_i)$ with $r_2< r_1 <r_3$. We may assume $s_1> s_2 \geq s_3$. Then $(r_1,s_2)\in \conv(y_1,y_2,y_3)\subseteq K$  and thus 
\begin{equation*}
(r_1,s_2)\in C_{z_1,0}^n\cap \dots \cap C_{z_k,0}^n\cap K. 
\end{equation*} 
\qed
\end{proof}

\begin{example}
For $n=2$ and $\La=\Z^2$, $\hat{V}_0^2$ is the algorithm suggested by Pavlidis in \cite{pavlidis}, which is multigrid convergent on the class of $r$-regular sets. Theorem \ref{nalg} shows that this algorithm is also multigrid convergent on the class of compact convex polygons with no interior angles of less than 45 degrees. 
\end{example}

\subsection{The Euler characteristic in higher dimensions}\label{EChigh}

The results of the previous sections allow us to ge\-ne\-ra\-lize the 2D case of Theorem \ref{main2}  and Corollary \ref{worst} to higher dimensions.

\begin{theorem}
For $d\geq 2$ and $q\leq d-2$, any local algorithm $\hat{V}_q$ for which $\lim_{a\to 0}E\hat{V}_q(P)$ exists for all $P\in \mathcal{P}^d$ has a worst case asymptotic relative bias of at least 100\% on $\mathcal{P}^d$. In particular, Theorem \ref{main2} holds.
\end{theorem}

The proof uses the fact that if $P=\bigoplus_{i=1}^d [0,v_i]$ with $v_1,\dots,v_d \in \R^d$ linearly independent, then
\begin{equation}\label{Vqbox}
V_q(P)=\sum_{1\leq i_1<\dots < i_q\leq d} \Ha^q\bigg(\bigoplus_{s=1}^q [0,v_{i_s}]\bigg).
\end{equation}
This follows because 
\begin{equation*}
\sum_{S\subseteq \{1,\dots,d\}\backslash \{i_1,\dots , i_q\}} \gamma \bigg(\sum_{i\in S} v_{i}+\bigoplus_{s=1}^q [0,v_{i_s}],P\bigg)=1.
\end{equation*}
%where $J_{i_1,\dots , i_q}= \{\sum_{i\in S} v_{i}\mid S\subseteq \{1,\dots,d\}\backslash \{i_1,\dots , i_q\}\}$.

\begin{proof}
First consider the case of the standard lattice $\Z^d$. 
Take $Q=P\oplus \bigoplus_{j=3}^d [0,e_j]$ where $e_1,\dots,e_d\in \R^d$ is the standard basis and $P\subseteq \text{lin}\{e_1,e_2\}\cong \R^2$ is as in \eqref{parallel}. Let $M:\R^d \to \R^d$ be a linear map taking $P$ to $[0,s_1e_1]\oplus[0,s_2e_2]$ and fixing $e_3,\dots ,e_d$. Then
\begin{equation}\label{limtransf}
\lim_{a\to 0}E\hat{V}_q^{a\Z^d}(Q)=\lim_{a\to 0} E\hat{V}_q^{aM(\Z^d)}(M(Q)).
\end{equation}
The right hand side is a polynomial in $s_1,s_2,t_3,\dots,t_d$. 

On the other hand, \eqref{Vqbox} yields
\begin{align*}
V_q(Q)={}&\sum_{\substack{S\subseteq \{3,\dots,d\},\\ |S|=q}} \prod_{i\in S} t_i + \sum_{\substack{S\subseteq \{3,\dots,d\},\\ |S|=q-1}} (s_1+s_2)\prod_{i\in S} t_i\\
&+  \sum_{\substack{S\subseteq \{3,\dots,d\},\\ |S|=q-2}} s_1s_2\sin \psi\prod_{i\in S} t_i.
\end{align*}

If \eqref{limtransf} contains monomials in $s_1,s_2,t_3,\dots,t_d$ of degree larger than $q$ or if it contains monomials of degree less than $q$, we can make the relative bias arbitrarily large just by scaling $Q$ up or down, respectively.

Otherwise, \eqref{limtransf} is homogeneous in $s_1,s_2,t_3,\dots,t_d$ of degree $q$, and the argument in the proof of Corollary~\ref{weights} shows that the weights may be assumed to be homogeneous of degree $q$. Observing that 
\begin{equation*}
(-h(M(B_{l_i}\oplus \check{W}_{l_i}),\pm e_i))^+\in \{0,1\}
\end{equation*}
 for $i=3,\dots,d$ and writing $s_i=t_i$ for $i=1,2$,  the proof of Theorem \ref{main1'} shows that
\begin{align*}
\lim_{a\to 0}{}& E\hat{V}_q^{aM(\Z^d)}(M(Q))\\
={}& \sum_{l_1,\dots,l_{d-q}=1}^{2^{n^d}-1} w_{l_1,\dots,l_{d-q}} \\
&\times \prod_{i \in I^{L}}\sum_{\eps_i\in {\pm 1}}(-h(M(B_{l_i}\oplus \check{W}_{l_i}),\eps_ie_i))^+
 \prod_{i\notin I^{L}} t_i\\
={}&\lim_{a\to 0}E\hat{V}_{0}^{\prime a\Z^d}(P)\sum_{\substack{S\subseteq \{3,\dots,d\},\\ |S|=q}} \prod_{i\in S} t_i\\
&+ \sum_{\substack{S\subseteq \{3,\dots,d\},\\ |S|=q-1}}(\beta_S^1(Q)s_1+\beta_{S}^2(Q)s_2) \prod_{i\in S} t_i\\ 
&+ \sum_{\substack{S\subseteq \{3,\dots,d\},\\ |S|=q-2}}\beta_{S}^{12}(Q)s_1s_2\prod_{i\in S} t_i
\end{align*}
%\end{align*}
where $I^L$ is as defined just before the statement of Lemma~\ref{polvol}, $\beta_S^1(Q),\beta_S^2(Q),\beta_S^{12}(Q)$ are certain numbers depending only on $(\varphi,\psi)$ and the chosen weights, and $V_0'$ is the estimator for $V_0$ in $\R^2$ with weights
\begin{align*}
w_l'={}&\sum_{3\leq i_1 < \dots < i_{d-q-2}\leq d }\sum_{c_3,\dots,c_{d-q}=-n+1}^{n-2} \\
&w(\pi^{-1}(B_l)\cap B_{i_1,c_1} \cap \dots \cap B_{i_{d-q-2},c_{d-q-2}})
\end{align*}
where $\pi : C_{0,0}^n(\R^d)\to C_{0,0}^n(\R^2)$ is the projection induced by $\R^d \to \text{lin}\{e_1,e_2\}$ and $B_{j,c}=C_{0,0}^n \cap H^-_{e_j,c}$ for $c=0,\dots,n-2$, while $B_{j,c}=C_{0,0}^n \cap H^-_{-e_j,c}$ for $c=-n+1,\dots,-1$. 

By Corollary \ref{worst}, $\lim_{a\to 0}E\hat{V}_{0}'(P)$ is either zero or can be made arbitrarily large by properly choosing $P$. Thus, the asymptotic worst case error can be made arbitrarily close to one or arbitrarily large, respectively, by choosing  $P$ first and then choosing $t_3,\dots , t_d$ small compared to  $s_1$ and $s_2$.

Now consider a general lattice $\La$. Choose a linear map $M: \R^d \to \R^d$ such that $M(\Z^d)=\La$. Then
\begin{equation*}
E\hat{V}_q^{a\Z^d}(Q) = E\hat{V}_q^{a\La}(M(Q)),
\end{equation*}
while 
\begin{align*}
V_q(M(Q))={}&\sum_{\substack{S\subseteq \{3,\dots,d\},\\ |S|=q}} \alpha_S \prod_{i\in S} t_i\\
 &+ \sum_{\substack{S\subseteq \{3,\dots,d\},\\ |S|=q-1}}(\alpha_S^1s_1+\alpha_S^2s_2) \prod_{i\in S} t_i\\
 &+\sum_{\substack{S\subseteq \{3,\dots,d\},\\ |S|=q-2}}\alpha_S^{12} s_1s_2 t_i
\end{align*}
where $\alpha_S$ depends only on $M$ while $\alpha_S^1,\alpha_S^2,\alpha_S^{12}$ may also depend on $(\varphi,\psi)$. Thus the general case follows as before by first choosing $P$ and then choosing $s_1,s_2$.\qed
\end{proof}

\section{Local estimators on the class of $r$-regular sets}\label{rreg}
We now move on to local digital algorithms applied to $r$-regular sets. The formal definition of $r$-regular sets is as follows:

\begin{definition}\label{defrreg}
$X\subseteq \R^d$ is called $r$-regular if for every $x\in \partial X$ there exist two balls $B_1,B_2\subseteq \R^d$ both of radius $r$ and containing $x$ with $B_1\subseteq X$ and $\indre (B_2) \subseteq \R^d \backslash X$. For $x\in \partial X$, $n(x)$ denotes the unique outward pointing normal vector.
\end{definition}

The purpose of this section is to prove Theorem \ref{mainrreg}. As we only consider local estimators with homogeneous weights, we may assume $w_{2^{n^d}-1}=0$, see \cite[Section~3]{am2}. 
The case of the surface area $2V_{d-1}$ is an easy consequence of the corresponding theorem for polytopes and the following formula due to Kiderlen and Rataj \cite[Theorem 5]{rataj}:
\begin{theorem}[Kiderlen, Rataj]\label{KRthm}
For any local estimator  $\hat{V}_{d-1}$ with homogeneous weights and $w_{2^{n^d}-1}=0$ and for any compact $r$-regular set $X\subseteq \R^d$:
\begin{align*}
\lim_{a\to 0} {}&E\hat{V}_{d-1}(X)\\
={}& \det(\La)^{-1}\sum_{l=1}^{2^{n^d}-2} w_l \int_{\partial X} (-h(B_l\oplus \check{W}_l,n))^+ d\Ha^{d-1}.
\end{align*}
\end{theorem}

\begin{proof}[Theorem \ref{mainrreg} for $V_{d-1}$]
Suppose $\hat{V}_{d-1}$ is given. By Corollary \ref{box}, we may choose $v_1,\dots, v_d\in \R^d$ orthogonal such that
\begin{equation*}
\lim_{a\to 0}E\hat{V}_{d-1}\bigg(\bigoplus_{i=1}^d [0,v_i]\bigg)\neq {V}_{d-1}\bigg(\bigoplus_{i=1}^d [0,v_i]\bigg).
\end{equation*}
Consider the $r$-regular set
\begin{equation*}
X(r)=B(r)\oplus \bigoplus_{i=1}^d [0,v_i].
\end{equation*}

Observe that
\begin{equation*}
\lim_{r\to 0} V_{d-1}(X(r)) = {V}_{d-1}\bigg(\bigoplus_{i=1}^d [0,t_iu_i]\bigg).
\end{equation*}
On the other hand, Theorem \ref{KRthm} yields
\begin{align*}
\lim_{a\to 0}{}&E\hat{V}_{d-1}(X(r))\\
 =&\det(\La)^{-1}\sum_{l=1}^{2^{n^d}-2} w_l \int_{\partial X(r)} (-h(B_l\oplus \check{W}_l,n))^+ d\Ha^{d-1}\\
 =& \lim_{a\to 0}E\hat{V}_{d-1}\left(\bigoplus_{i=1}^d [0,v_i]\right)\\
 & + \det(\La)^{-1}\sum_{l=1}^{2^{n^d}-2} w_l \int_{Y}(- h(B_l\oplus \check{W}_l,n))^+ d\Ha^{d-1}
\end{align*}
where
\begin{equation*}
Y= X \backslash \bigcup_{F\in \mathcal{F}_{d-1}(\bigoplus_{i=1}^d [0,v_i])} (F +ru_{i_1^F}).
\end{equation*}

Since each $h(B_l\oplus \check{W}_l,n)$ is bounded for $n\in S^{d-1}$ and $\lim_{r\to 0}\Ha^{d-1}(Y) = 0$, it follows that
\begin{equation*}
\lim_{r\to 0}\lim_{a\to 0} E\hat{V}_{d-1}(X(r)) = \lim_{a\to 0}E\hat{V}_{d-1}\bigg(\bigoplus_{i=1}^d [0,v_i]\bigg). 
\end{equation*}
In particular, 
\begin{equation*}
 \lim_{a\to 0}E\hat{V}_{d-1}(X(r)) \neq V_{d-1}(X(r))
\end{equation*}
when $r$ is sufficiently small.\qed
\end{proof}

It follows from the definition of $r$-regularity that the boundary of an $r$-regular set $X$ is a $C^1$ manifold. The normal vector field $n$ is almost everywhere dif\-fe\-ren\-tiable on $\partial X$, see \cite{federer}. In particular, the second fundamental form $\II_x $  is defined on the tangent space $T_x \partial X $ if $n$ is dif\-fe\-ren\-tiable at $x$. Define $Q_x$ to be the quadratic form  on $T_x \partial X \oplus \text{lin}\{n(x)\}=\R^d$ given by
\begin{equation*}
Q_x(\alpha, tn(x)) = -\II_x(\alpha) + \tr(\II_x)t^2.
\end{equation*}

For a finite set $S\subseteq \R^d$, define
\begin{align*}
\II^+_x(S)&= \max \{\II_x(s)\mid s\in S,\, h(S,n)=\langle s,n\rangle \}, \\
\II^-_x(S)&= \min \{\II_x(s)\mid s\in S,\, h(S,-n)=\langle s,-n\rangle \}.
\end{align*}
Here $\II_x(s)$ for $s\in \R^d $ means $\II_x(\pi_x(s))$ where $\pi_x$ is the projection onto $T_x\partial X$. If $s^\pm\in S$ are such that $\II^\pm_x(S)=\II^\pm_x(s^\pm)$, define
\begin{equation*}
Q^{\pm}_x(S)=Q_x(s^\pm).
\end{equation*}
The following formula is shown in \cite{am2}:
\begin{theorem} \label{Q}
For a local estimator $\hat{V}_{d-2}$ with homogeneous weights and $w_{2^{n^d}-1}=0$ and  an $r$-regular set $X$,
\begin{align*}
\lim_{a\to 0} {}&E\hat{V}_{d-2}(X) =\det(\La)^{-1}\frac{1}{2} \sum_{l=1}^{2^d-2}w_l \\ \nonumber
&\times \int_{\partial X} (Q^+(B_l)-Q^-(W_l))\delta_{l}(n)  \\
&- (\II^+(B_l)-\II^-(W_l))^+\mathds{1}_{\{h(B_l\oplus \check{W}_l,n)=0\}}d\Ha^{d-1}.
\end{align*}
\end{theorem}

The proof of Theorem \ref{mainrreg} for $V_{d-2}$ follows from this:
\begin{proof}[Theorem \ref{mainrreg} for $V_{d-2}$]
We first introduce the sets that will serve as counter examples. For $0<r<R$ and $\theta \in (0,\pi)$, let 
\begin{equation*}
T(R,r) = B(r) \oplus B^{d-1}((R-r)\sin \theta)
\end{equation*}
where $B^{d-1}(s)$ is the ball of radius $s $ in $\lin(e_1,\dots ,e_{d-1})$. We then consider $r$-regular sets of the form
\begin{align*}
X{}&(R,r)\\
&=(B(R)\cap H^-_{R\cos \theta ,e_d }) \cup (T(R,r)  + (R-r)\cos \theta e_d), 
\end{align*}
see Figure \ref{fig2}.

\begin{figure}
\setlength{\unitlength}{1cm}
\begin{equation*}
\begin{picture}(6,4)
\qbezier(1,2)(1.1,0.1)(3,0)
\qbezier(3,0)(4.9,0.1)(5,2)
\qbezier(1,2)(1,2.9)(1.6,3.4)
\qbezier(5,2)(5,2.9)(4.4,3.4)
\put(1.9,3){\circle{1}}
\put(4.1,3){\circle{1}}
\put(1.9,3.5){\line(1,0){2.2}}
\put(1.6,3.4){\line(1,0){2.8}}
\put(3,2){\line(1,0){2}}
\put(3,2){\line(1,1){1.4}}
\put(3,2){\line(0,1){1.5}}
\put(4,1.6){$R$}
%\put(3.6,2.2){$R$}
\put(3.1,2.5){$\theta$}
\qbezier(3,2.3)(3.1,2.3)(3.2,2.2)
\put(1.9,3){\line(1,0){0.5}}
\put(2,2.7){$r$}
\put(2.9,3.7){$S_3$}
\put(0.8,3.7){$S_2$}
\put(0.5,1.5){$S_1$}
\put(1.2,3.8){\vector(2,-1){0.5}}
\end{picture}
\end{equation*}
\caption{Sketch of the set $X(R,r)$.}
\label{fig2}
\end{figure}
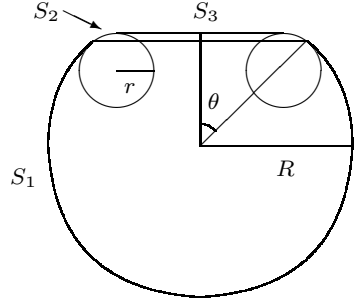
%\begin{equation*}
%\setlength{\unitlength}{1mm}
%\begin{picture}(44,40)
%\qbezier(3,3)(2,3)(2,2)
%\end{picture}
%\end{equation*}

Choose a rotation $\rho \in SO(d)$ taking $e_d$ to a point in $S^{d-1}\backslash D$ where $D$ is as in the proof of Theorem \ref{main1'} and consider $\rho(X(R,r))$. Then 
\begin{equation*}
\hat{V}_{d-2}^{a \La_c}(\rho X(R,r) ) =\hat{V}_{d-2}^{a\rho^{-1}\La_c}(X(R,r)),
\end{equation*}
so by possibly changing the lattice, we may assume that $\rho=I$ and $e_d \in S^{d-1}\backslash D$.  Let $U \subseteq S^{d-1}\backslash D$ be the connected component containing $e_d$. This is open in $S^{d-1}$. 

This ensures that 
\begin{equation*}
\Ha^{d-1}(x\in \partial X\mid h(B_l\oplus\check{W}_l,n(x))=0)=0
\end{equation*}
 for all $l$ and we may thus ignore the last line of the formula in Theorem \ref{Q}.

If $n \notin D$, there exist unique vectors $b_l(n)\in B_l$ and $w_l(n)\in W_l$ such that $Q^+_x(B_l)=Q_x(b_l(n))$ and $Q^-_x(W_l)=Q_x(w_l(n))$ for all $x$ with $n(x)=n$.
This defines functions 
\begin{equation*}
\beta_l,\omega_l : S^{d-1}\backslash D \to C_{0,0}^n.
\end{equation*}
Note that these are locally constant and so is the indicator function $\delta_{l}$ on $S^{d-1}\backslash D$.

Let $\eps_1,\dots ,\eps_{d-1}\in T\partial X(R,r)$ denote the principal directions corresponding to the principal curvatures $k_1,\dots ,k_{d-1}$. Since $e_d\notin D$,  $n(x)\in  S^{d-1}\backslash D$ for almost all $x\in \partial X(R,r)$ and for such $x$
\begin{align*}
Q^+_x(B_l)-Q^-_x(W_l)={}& \sum_{j=1}^{d-1}k_j(-\langle \beta_l(n), \eps_j\rangle^2 + \langle \beta_l(n),n\rangle^2\\
& +\langle \omega_l(n), \eps_j \rangle^2 - \langle \omega_l(n),n\rangle^2).
\end{align*}

Observe that $\partial X(R,r)$ is the disjoint union of three sets $S_1$, $S_2$, and $S_3$ where
\begin{align*}
S_1 &= (\partial B(R))\cap H^-_{R\cos \theta ,e_d },
\\
S_2&= (\partial T(R,r)  + (R-r)\cos \theta e_d) \backslash  (H^-_{R\cos \theta ,e_d }\cup S_3),
\\
S_3&= B^{d-1}((R-r) \sin \theta) + ((R-r)\cos \theta+r)e_d.
\end{align*}
On $S_3$, $k_1=\dots =k_{d-1}=0$ and thus $Q$ vanishes on $S_3$.  

Parametrize $S_1$ by $g_1:S^{d-2}\times (\theta,\pi)\to S_1$. Identifying
 $S^{d-2}$ with the unit sphere in $\lin(e_1, \dots, e_{d-1})\subseteq \R^d$,
\begin{equation*}
g_1(u,\varphi)= R(\sin\varphi u  + \cos \varphi e_d).
\end{equation*}
Similarly, parametrize $S_2$ by $g_2:S^{d-2}\times (0,\theta) \to S_2$ where 
\begin{equation*}
g_2(u,\varphi)= (R-r)\sin\theta u  + r\cos \varphi e_d. 
\end{equation*}
Note that on both $S_1$ and $S_2$,
\begin{align*}
&n(u,\varphi)= \sin\varphi u + \cos \varphi e_d,\\
&\eps_{d-1}(u,\varphi) = - \cos\varphi u + \sin \varphi e_d,\\
&\eps_j(u,\varphi) = \eps_j'(u),
\end{align*}
for $j=1,\dots, d-2$, where $\eps_j'(u)$ are the principal directions on $S^{d-2}$.

On $S_1$,
\begin{align*}
&k_j=\frac{1}{R}\\
&\int_{S_1} f = \int_{\theta}^{\pi} \int_{S^{d-2}} f(u,\varphi) R^{d-1} \sin^{d-2} \varphi \Ha^{d-2}(du) d\varphi 
\end{align*}
for all $j=1,\dots, d-1$ and any integrable function $f$. On $S_2$,
\begin{align*}
&k_{d-1}=\frac{1}{r}\\
&k_j(\varphi)= \frac{\sin\varphi}{(R-r)\sin \theta +r\sin\varphi}\\
&\int_{S_2} f = \int^{\theta}_{0} \int_{S^{d-2}} f(u,\varphi)\\
&\quad \qquad \times r((R-r)\sin \theta +r\sin\varphi)^{d-2} \Ha^{d-2}(du) d\varphi 
\end{align*}
for $j=1,\dots, d-2$ and any integrable function $f$.

Define $F_1,F_2 : (0,\pi) \to \R$ by
\begin{align*}
F_1(\varphi) ={}&\det(\La)^{-1} \frac{1}{2}\sum_{l=1}^{2^{n^d}-2} w_l \int_{S^{d-2}}\\
&\Big(-\langle \beta_l, \eps_{d-1}(u,\varphi)\rangle^2 + \langle  \beta_l,n(u,\varphi) \rangle^2\\
& \, +\langle \omega_l, \eps_{d-1}(u,\varphi)\rangle^2 - \langle  \omega_l , n(u,\varphi) \rangle^2\Big)\\
&\times \delta_{l}(n(u,\varphi))\Ha^{d-2}(du),\\
F_2 (\varphi)={}& \det(\La)^{-1}\frac{1}{2} \sum_{l=1}^{2^{n^d}-2} w_l \sum_{j=1}^{d-2} \int_{S^{d-2}}  \\
&\Big(-\langle \beta_l, \eps_{j}(u,\varphi) \rangle^2 + \langle  \beta_l,n(u,\varphi) \rangle^2\\
& +\langle \omega_l, \eps_{j}(u,\varphi)\rangle^2 - \langle  \omega_l, n(u,\varphi) \rangle^2\Big)\\ &\times \delta_{l}(n(u,\varphi))\Ha^{d-2}(du).
\end{align*}

Theorem \ref{Q} yields 
\begin{align}%\nonumber
\lim_{a\to 0}E\hat{V}_{d-2}(X(R,r)) %&= \int_0^\pi \bigg(k_{d-1}F_1 + \sum_{j=1}^{d-2} k_j F_2\bigg) d\varphi \\
&=I_1+I_3,\label{estint}
\end{align}
where 
\begin{align*}
I_1 ={}&\int_{S_1}\det(\La)^{-1} \frac{1}{2}\sum_{l=1}^{2^{n^d}-2} w_l(Q^+_x(B_l)-Q^-_x(W_l))\\
&\times \delta_l(n)d\Ha^{d-1} \\
={}&\int_{\theta}^{\pi}\bigg(k_{d-1}F_1 + \sum_{j=1}^{d-2}k_jF_2\bigg) R^{d-1}\sin^{d-2}\varphi d\varphi \\
%=\int_{S_1}\bigg(k_{d-1}F_1 + \sum_{j=1}^{d-2}k_jF_2\bigg) d\Ha^{d-1} \\
={}&R^{d-2}\int^{\pi}_{\theta} (F_1(\varphi )+(d-2)F_2(\varphi))\sin^{d-2}\varphi d\varphi,
\end{align*}
and
\begin{align*}
I_3={}&\int_{S_2}\det(\La)^{-1} \frac{1}{2}\sum_{l=1}^{2^{n^d}-2} w_l(Q^+_x(B_l)-Q^-_x(W_l))\\
&\times \delta_l(n)d\Ha^{d-1} \\
={}&\int_{0}^{\theta}\bigg(k_{d-1}F_1 + \sum_{j=1}^{d-2}k_jF_2\bigg)\\
&\times r({(R-r)\sin\theta+r\sin\varphi})^{d-2}d\varphi \\
%{}&\int_{ S_2} \bigg(k_{d-1}F_1 + \sum_{j=1}^{d-2}k_jF_2\bigg) d\Ha^{d-1}\\ \nonumber
={}& \int_{0}^{\theta} \left(\frac{1}{r}F_1(\varphi) + \left( \frac{(d-2)\sin \varphi}{(R-r)\sin\theta+r\sin\varphi}\right)F_2(\varphi)\right)\\
&\times r({(R-r)\sin\theta+r\sin\varphi})^{d-2}d\varphi\\ \nonumber
={}&\int_{0}^{\theta} (({(R-r)\sin\theta+r\sin\varphi})^{d-2}F_1(\varphi) + (d-2)\\
&\times r{\sin \varphi}({(R-r)\sin\theta +r\sin\varphi})^{d-3}F_2(\varphi))d\varphi\\ \nonumber
={}&R^{d-2}\sin^{d-2}\theta \int_{0}^{\theta}  F_1(\varphi) d\varphi + r\tilde{p}(r,\theta).
\end{align*}
Here $\tilde{p}$ is a polynomial in $r$ with coefficients depending only on $\theta$ and $R$.

On the other hand,
\begin{align}%\nonumber
V_{d-2}(X(R,r))%&=\frac{1}{2\pi}\int_{S_1\cup S_2} \bigg(k_{d-1} + \sum_{j=1}^{d-2}k_j\bigg)d\Ha^{d-1}\\
&= \frac{1}{2\pi}(I_2+I_4)\label{trueint}
\end{align}
where
\begin{align*}
I_2&=\int_{S_1}\bigg(k_{d-1} + \sum_{j=1}^{d-2}k_j \bigg) d\Ha^{d-1} \\
&=\Ha^{d-2}(S^{d-2})(d-1)R^{d-2}\int^{\pi}_{\theta} \sin^{d-2}\varphi d\varphi
\end{align*}
and
\begin{align*}
I_4={}&\int_{ S_2} \bigg(k_{d-1} + \sum_{j=1}^{d-2}k_j\bigg) d\Ha^{d-1}\\ \nonumber
={}& \Ha^{d-2}(S^{d-2}) \int_{0}^{\theta} r({(R-r)\sin\theta +r\sin\varphi})^{d-2}  \\
& \times  \left(\frac{1}{r} + (d-2)\left( \frac{\sin \varphi}{(R-r)\sin\theta+r\sin\varphi}\right)\right)d\varphi\\ \nonumber
={}&\Ha^{d-2}(S^{d-2})\int_{0}^{\theta} (({(R-r)\sin\theta+r\sin\varphi})^{d-2} \\
&+ (d-2)r{\sin \varphi}({(R-r)\sin\theta+r\sin\varphi})^{d-3})d\varphi\\ \nonumber
={}&\Ha^{d-2}(S^{d-2})R^{d-2} \sin^{d-2}\theta \int_{0}^{\theta}1d\varphi + r{p}(r,\theta).
\end{align*}
Again ${p}$ is a polynomial in $r$ with coefficients depending only on $\theta$ and $R$.

Since $\hat{V}_{d-2}$ is asymptotically unbiased, \eqref{estint} must equal \eqref{trueint}, i.e.
\begin{equation*}
I_1+I_3 = \tfrac{1}{2\pi}(I_2+I_4).
\end{equation*}
This must hold for all $0<r<R$, so letting $r\to0 $ shows that
\begin{align}\nonumber
\int^{\pi}_{\theta} {}&(F_1(\varphi)+(d-2)F_2(\varphi))\sin^{d-2}\varphi d\varphi \\\label{musthold}
&+ \sin^{d-2} \theta\int_{0}^{\theta}  F_1(\varphi) d\varphi\\ \nonumber
={}& (d-1)\frac{1}{2\pi}\int^{\pi}_{\theta}   \sin^{d-2}\varphi d\varphi  +\sin^{d-2}\theta \frac{1}{2\pi} \int_{0}^{\theta} 1 d\varphi 
\end{align}
holds for all $\theta\in (0, \pi)$.

The assumption $e_d\in U$ ensures that for small values of $\varphi$, $n(u,\varphi)\in U$ for all $u\in S^{d-2}$, and hence all $b_l$, $w_l$, and $\delta_{l}$ are constants. This shows that $F_1$ and $F_2$ are continuous for small $\varphi$. In fact, a direct computation shows that for such small $\varphi$,
\begin{align*}
F_1(\varphi)&=K_1(\sin^2\varphi- \cos^2 \varphi)+K_2\sin\varphi \cos\varphi,\\
F_2(\varphi)&= K_3 + K_4\sin^2\varphi +K_5\cos^2\varphi+ K_6\sin\varphi \cos\varphi, 
\end{align*}
where $K_1,\dots , K_6\in \R$ are certain constants. In par\-ti\-cu\-lar, \eqref{musthold} may be differentiated with respect to $\theta$ for small values of $\theta$. This yields
\begin{align}\nonumber
({}&d-2)\bigg( -F_2(\theta)\sin^{d-2}\theta  +\cos\theta \sin^{d-3}\theta\int_0^{\theta}  F_1(\varphi)d\varphi \bigg) \\ 
&= \frac{(d-2)\Ha^{d-2}(S^{d-2})}{2\pi}(-\sin^{d-2}\theta + \theta \cos \theta \sin^{d-3}\theta)\label{must}
\end{align}
for $\theta$ small.

Since $d-2\neq 0$, \eqref{must} shows that $\theta \cos \theta \sin^{d-3}\theta$ must be a polynomial in $\cos \theta$ and $\sin \theta$, which is a contradiction. \qed
\end{proof}

\begin{acknowledgements}
The author is supported by Centre for Stochastic Geometry and Advanced Bioimaging, funded by the Villum Foundation.
The author is most thankful to Markus Kiderlen for suggesting this problem in the first place and for useful input along the way. 
\end{acknowledgements}

\end{document}